\documentclass{amsart}
\usepackage{graphicx} 
\numberwithin{equation}{section}
\usepackage{amsmath}
\usepackage{amssymb}
\usepackage{amsthm}
\usepackage{color}
\usepackage{enumerate}
\theoremstyle{definition}
\newtheorem{thm}{Theorem}[section]
\newtheorem{prop}[thm]{Proposition}
\newtheorem{rem}[thm]{Remark}
\newtheorem{lemma}[thm]{Lemma}
\newtheorem{definition}[thm]{Definition}
\newtheorem{example}[thm]{Example}

\newcommand{\exponent}{\gamma}
\newcommand{\RP}{\mathbb{P}}
\newcommand{\RD}{\theta}
\newcommand{\RV}{\omega}
\newcommand{\RS}{\Omega}
\newcommand{\CF}{F_A^{\infty}}
\newcommand{\CE}{E_A^{\infty}}
\newcommand{\TO}{\mathcal{L}_{f,\RV}}
\newcommand{\TOF}{\mathcal{L}_{f,F,\RV}}

\newcommand{\CFspace}{C^b(\CF)}
\newcommand{\HFspace}{H^b(\CF)}

\newcommand{\TOk}{\mathcal{L}_{f,\RV_{k-1}}}

\newcommand{\imsF}{\mathcal{I}_{\RP}(\CF\times\RS)}

\newcommand{\h}{H_{\exponent,\RS}}

\newcommand{\cc}{s_{\Phi}}
\newcommand{\dis}{\RS_{d,f}}
\newcommand{\bd}{B_{f}}
\newcommand{\pr}{\RS_{p,f}}
\newcommand{\fin}{\RS_{b,f}}

\newcommand{\K}{K_{f}}
\newcommand{\cb}{\underline{\RS}_{b}}
\newcommand{\mFs}{m_{F,\RV}^s}

\newcommand{\loc}{\RS_{F,\text{loc}}}

\newcommand{\diR}{\RS_{d,\Phi}}
\newcommand{\preintervalF}{\RS_{p,F,\Phi}}

\newcommand{\preinterval}{\RS_{p,\Phi}}
\newcommand{\OSC}{O_F}
\newcommand{\lowb}{\underline{\RS}_{b,\Phi}}

\newcommand{\finFPhi}{\RS_{b,\zeta_F,\Phi}}
\newcommand{\finEPhi}{\RS_{b,\zeta,\Phi}}

\newcommand{\tail}{\RS_{f,\text{tail}}}
\newcommand{\tailR}{\RS_{\Phi,F,\text{tail}}}
\newcommand{\tailRE}{\RS_{\Phi,\text{tail}}}
\newcommand{\RU}{Fin_{\text{RU}}}

\title[Multifractal analysis of Lyapunov spectrum for RCGDMS]{Multifractal analysis of the Lyapunov exponent for random graph directed Markov systems}
\author{Yuya Arima}
\date{\today}

\address{Graduate School of Mathematics, Nagoya University,
Furocho, Chikusaku, Nagoya, 464-8602, JAPAN} 
\email{yuya.arima.c0@math.nagoya-u.ac.jp}

\subjclass[2020]{37C45, 37D35, 37E05, 37H12}
\thanks{{\it Keywords}: multifractal analysis, thermodynamics formalism, random dynamical systems, countable Markov shift}

\begin{document}
\begin{abstract}
In this paper, we perform the multifractal analysis of the Lyapunov exponent for random conformal graph directed Markov systems introduced by Roy and Urba\'nski (2011). We also generalize Bowen's formula for the limit set of a random conformal graph directed Markov system established by Roy and Urba\'nski. 
To do this, we develop several refined properties of random finitely primitive countable Markov shifts.
\end{abstract}

\maketitle

\section{Introduction}
Let us consider a differentiable dynamical system $T:X\rightarrow X$, where for simplicity $X$ is a subset of the $d$-dimensional Euclidean space $\mathbb{R}^d$ and $T$ is conformal. For $x\in X$ we denote by $|T'(x)|$ the Euclidean norm  of the derivative of $T$ at $x$.
In this setting, the Lyapunov exponent of $T$ at $x\in X$ is given by 
\[
\chi(x):=\lim_{n\to\infty}\frac{1}{n}\log |(T^n)'(x)|
\]
whenever the limit exists. 
By Birkhoff's ergodic theorem, if there exists a $T$-invariant  Borel probability measure $\mu$ on $X$ such that $\mu$ is ergodic with respect to $\mu$ and $\int |\log |T'||d\mu<\infty$ then for $\mu$-a.e. $x\in X$ we have $\chi(x)=\int \log |T'|d\mu$. Hence, the set of points $x\in X$ satisfying $\chi(x)$ does not converge to $\int \log |T'|d\mu$  is negligible with respect to $\mu$. However, there is still a possibility that for $\alpha\neq \int \log |T'|d\mu$ the set $L(\alpha)$ of points $x$ that $\chi(x)=\alpha$ is not an empty set and might be a big set from another point of view. Here, the following natural questions arise: What is the typical or exceptional Lyapunov exponent? What values can the Lyapunov exponent take?
To answer these questions, Eckmann and Procaccia \cite{eckmann1986fluctuations} introduced the notion of the Lyapunov spectrum, that is, the map $\alpha\mapsto \dim_H(L(\alpha))$, where $\dim_H(L(\alpha))$ denotes the Hausdorff dimension of the set $L(\alpha)$ with respect to the Euclidean metric on $X$. Inspired by this work, Weiss \cite{LyapunovWeiss} conducted a rigorous study of the Lyapunov spectrum for conformal expanding maps with finitely many branches, employing tools developed by Pesin and Weiss \cite{bookpasin, pesinweisslocal}. After that, many mathematicians have studied the Lyapunov spectrum in the context of various dynamical systems. We refer to \cite{Przytycki, multimodal, Rams, Iommibackward, notconcave, Johannes, infractionPollicot, Kessebhomological, Kesseb,  Pollicott} for results on Lyapunov spectrum. In particular,  \cite{LyapunovWeiss} is generalized to the context of interval maps with countable many branches and infinite topological entropy by Pollicott and Weiss \cite{Pollicott} and Kessenb\"ohmer and Stratmann \cite{Kesseb}. In this paper, we study the Lyapunov spectrum for random conformal graph directed Markov systems introduced by Roy and Urba\'nski \cite{RGDMS}.

\subsection{The statement of our main theorem}
 Let $V$ be a finite set of vertices and let $E$ be a countable infinite set of directed edges. We consider two functions $i:E\rightarrow V$ and $t:E\rightarrow V$ by setting $i(e)$ to be the initial vertex of $e$ and $t(e)$ to be the terminal vertex of $e$ for each directed edge $e\in E$. Let $A:E\times E\rightarrow \{0,1\}$ be a incidence matrix defined by $A_{e_1,e_2}=1$ if and only if $t(e_1)=i(e_2)$ and let $\{X_v\}_{v\in V}$ be a family of non-empty compact subsets  of a common Euclidean space $\mathbb{R}^d$ ($d\geq 1$). A random graph directed Markov system (RGDMS) 
 \[
 \Phi:=(\RD:\RS\rightarrow \RS, \{\RV\mapsto\phi_{e,\RV}\}_{e\in E})
 \]
 is generated by an invertible bimeasurable ergodic measure-preserving map $\RD:\RS\rightarrow \RS$  with respect to a complete probability space $(\RS,\mathcal{F},\RP)$ and one-to-one contractions $\phi_{e,\RV}: X_{t(e)}\rightarrow X_{i(e)}$ with a common contraction ratio $\cc$ (see Definition \ref{def RGDMS} for the precise definition of RGDMSs). Furthermore, we assume that $\Phi$ is conformal (see Definition \ref{def CRGDMS} for the precise definition of conformal RGDMSs). 
 We set $E_A^{\infty}:=\{\tau\in E^{\mathbb{N}\cup\{0\}}:A_{\tau_{n-1},\tau_n}=1,\ n\in\mathbb{N}\}$.
 We denote by $\{\RV\mapsto\pi_\RV\}$ the random coding maps and by $\{\RV\mapsto J_\RV:=\pi_{\RV}(\CE)\}$ the random limit sets (see (\ref{eq def coding map}) for the definition of the random coding maps). In this section, we assume that $\Phi$ satisfies the random boundary separation condition with over finite subedges (see Definition \ref{def RBSCF}). Under this assumption, for all $\RV\in \RS$ the coding map $\pi_\RV:\CE\rightarrow J_\RV$ becomes a bijection.
 We define the family of expanding maps $\{T_{\RV,\Phi}: J_\RV\rightarrow J_{\RD(\RV)}\}_{\RV\in\RS}$ derived from $\Phi$  by 
\[
T_{\RV,\Phi}|_{J_\RV\cap\phi_{e,\RV}(X_{t(e)})}=\phi_{e,\RV}^{-1}
\text{ and } T_{\RV,\Phi}^n:=T_{\RD^{n-1}(\RV),\Phi}\circ\cdots\circ T_{\RV,\Phi}
\ (n\in\mathbb{N}).
\]
In this setting, we define the Lyapunov exponent of $\{T_{\RV,\Phi}: J_\RV\rightarrow J_{\RD(\RV)}\}_{\RV\in\RS}$ at $x\in J_\RV$ ($\RV\in\RS$)  by  
\[
\chi_\RV(x):=\lim_{n\to\infty}\frac{1}{n}\log|(T_{\RV,\Phi}^n)'(x)|
\]
when the limit exists. For $\RV\in\RS$ we also define the level set given by 
\[
L_\RV(\beta):=\left\{
x\in J_\RV:\chi_\RV(x)=\beta
\right\}
\ (\beta\in\mathbb{R}).
\]
For $\RV\in\RS$, using these level sets $L_\RV(\alpha)$ ($\alpha\in \mathbb{R}$), we obtain the following multifractal decomposition of $J_\RV$:
\[
J_\RV:=\bigcup_{\alpha\in \mathbb{R}}L_{\RV}(\alpha)\cup L_\RV',
\]
where $L_\RV'$ the set of points $x\in J_\RV$ for which the limit $\lim_{n\to\infty}(1/n)\log|(T_{\RV,\Phi}^n)'(x)|$ does not exist.
Finally, we define the Lyapunov spectrum $\{\RV\mapsto l_\RV\}$ by 
\[
l_\RV:\mathbb{R}\rightarrow [0,d],\ l_\RV(\beta):=\dim_H(L_\RV(\beta)).
\]

Let $Fin$ be the set of points $s\in\mathbb{R}$ such that the potential $(\tau,\RV)\mapsto s\zeta(\tau,\RV):= -s\log |T'_{\RV,\Phi}(\pi_\RV(\tau))|$ satisfies the summable condition in Definition \ref{def summabel}.
There exists $s_{\infty}\in\mathbb{R}\cup\{-\infty\}$ such that $Fin^{\circ}:=\text{Int}({Fin})=(s_\infty,\infty)$, where $\text{Int}(Fin)$ denotes the set of interior points of $Fin$.
If $E$ is a finite set then $s_\infty=-\infty$.
For $s\in {Fin}$ we denote by $p(s)$ the relative topological pressure for the potential $s\zeta$ (see Proposition \ref{prop definition of the pressure}) and for $s\notin {Fin}$ we set $p(s)=\infty$. 
A conformal RGDMS $\Phi$ is said to be normal (see \cite[(3.21)]{RGDMS}) if 
for each $e\in E$ there exist $M_e\in(0,s_{\Phi}]$  
such that for $\RP$-a.e. $\RV\in\RS$ we have
\begin{align}\label{eq not super expanding}
    \inf \{|\phi_{e,\RV}'(x)|:x\in X_{t(e)}\}>M_e.
\end{align}
We say that $\Phi$ is cofinitely regular if $p(s_\infty)=\infty$.
By the convexity of $s\mapsto p(s)$ on $Fin$, for each $s\in Fin^\circ$ the right-hand derivative $p'_+(s)$ of the function $p$ at $s$ exists.
We set
 $p'(s_\infty):=\lim_{s\to s_\infty+0}p'_-(s)$ and $p'(\infty):=\lim_{s\to+\infty}p'_+(s)$.
\begin{thm}\label{thm main}
Let $\Phi$ be a cofinitely regular normal random conformal graph directed Markov system satisfying the random boundary separation condition with over finite subedges. Then there exists $\RS_{\text{Ly}}\subset \RS$ such that $\RP(\RS_{\text{Ly}})=1$ and for all $\RV\in\RS_{\text{Ly}}$ and $\beta\in (-p'(+\infty),-p'(s_\infty))$ we have
\[
l_\RV(\beta)=\frac{1}{\beta}\inf \{\beta s+p(s):s\in\mathbb{R}\}.
\] 
Moreover, for all $\beta\notin[-p'(+\infty),-p'(s_\infty)]$ we have $L_\RV(\beta)=\emptyset$.
\end{thm}

In the context of random dynamical systems coded by random compact subshifts of finite type, various dimension spectra have been studied by performing multifractal analysis. We refer to \cite{Kiferfractal,mayer2011distance, Fengmultifractal, Lin, Zhihui} and the references therein. Roy and Urba\'nski established Bowen's formula with respect to the relative topological pressure function defined by (\ref{def RU pressure}) for the limit set generated by a normal random conformal graph directed Markov system with countable infinitely many edges. However, to our knowledge, there is no known result on the multifractal analysis for random dynamical systems coded by random non-compact subshifts of finite type.


In order to show the main theorem, we first extend the 
thermodynamic formalism for random finitely primitive countable Markov shift developed in \cite{Denkerkiferstadlbauer, RGDMS, StaOnrandomtop, Stacoupling}. Our main result on the thermodynamic formalism for random finitely primitive countable Markov shift is the construction of fiberwise multifractal measures.  By \cite{RGDMS, StaOnrandomtop}, for $t\in Fin$ there exist $\RS_t\subset \RS$ and random conformal measures $\{\widetilde \nu_{\RV}^t\}_{\RV\in\RS}$ such that $\RP(\RS_t)=1$ and for all $\RV\in \RS_t$ the measures $\widetilde \nu_{\RV}^t$ and $\widetilde \nu_{\RD(\RV)}^t$ satisfy the conformal property (see (\ref{eq def conformal measure})). On the other hand, in our main theorem, we take the full-measure set $\RS_{\text{Ly}}$ in a way that does not depend on $\beta\in (-p'(+\infty),-p'(s_\infty))$. 
This choice of the full-measure set $\RS_{\text{Ly}}$ enables us to express the function $l_\RV$ ($\RV\in\RS_{\text{Ly}}$) in terms of the relative topological pressure.
To obtain $\RS_{\text{Ly}}$, we need to take $\RS_t$ in a way that does not depend on $t\in Fin$. However, from the construction of random conformal measures in \cite{StaOnrandomtop} and \cite{RGDMS} using Crauel's relative Prohorov theorem \cite[Theorem 4.4]{crauel2002random}, it is difficult to take $\RS_t$ in a way that does not depend on $t\in Fin$. Therefore, we use an approach different from \cite{StaOnrandomtop} and \cite{RGDMS} to contract fiberwise multifractal measures. The approach is that using the classical Prohorov theorem, for $\RV$ in some full-measure set we construct fiberwise multifractal measures along the orbit $\{\RD^{n}(\RV)\}_{n\in\mathbb{N}\cup\{0\}}$ (see Theorem \ref{thm conformal measure}).


We establish Bowen's formula with respect to the relative topological pressure defined first in \cite{Denkerkiferstadlbauer} (see Proposition \ref{prop definition of the pressure}) for the limit set generated by a normal random conformal graph directed Markov system
(see Theorem \ref{Thm Bowen formula}).
By the definitions of the relative topological pressure defined in Proposition \ref{prop definition of the pressure} and the relative topological pressure defined in \cite{RGDMS} (see  (\ref{def RU pressure})), we can see that our version of Bowen's formula generalizes the one obtained in \cite{RGDMS} for the limit set generated by normal random conformal graph directed Markov systems with the evenly varying assumption (see Definition \ref{def evenly varing} for the definition).
We remark that although the statement of \cite[Theorem 3.26]{RGDMS} does not assume the evenly varying assumption, this assumption seems to play an important role in its proof (see Remark \ref{rem evenly varying}).  
We provide an example of a normal random conformal graph directed Markov system that does not satisfy the evenly varying assumption, for which our Bowen's formula holds, while Bowen's formula established in \cite{RGDMS} fails to hold. (see Example \ref{example Bowen formula}). 

\subsection{Plan of the paper}
In Section \ref{section Basic definitions on random countable Markov shift}, we give basic definitions and known results on random countable Markov shifts. 
In Section \ref{section Several forms of the relative topological pressure}, we discuss equivalent definitions of the relative topological pressure. 
 Section \ref{section Compact approximation property of the relative topological pressure} is devoted to proving the compact approximation property of the relative topological pressure. Section \ref{section Multifractal random conformal measures} we construct fiberwise multifractal measures. 
In Section \ref{section Conformal random graph directed Markov systems (CRGDMSs)}, we first give the definition of random graph detected Markov system and then prove Bowen's formula. 
In Section \ref{section multifractal}, the proof of our main theorem is given.

\section{Random countable Markov shifts}

\subsection{Preliminaries on random countable Markov shifts }\label{section Basic definitions on random countable Markov shift}
Let $E$ be a non-empty countable set and let $A:E\times E\rightarrow \{0,1\}$ be a matrix. We define the shift space by 
$
    E_A^\infty:=\left\{\tau\in E^{\mathbb{N}\cup\{0\}}:A_{\tau_{i-1},\tau_{i}}=1,\ i\in\mathbb{N}\right\}.$
For $n\in\mathbb{N}$ a string $\tau_0\tau_1\cdots\tau_{n-1}\in E^n$ with length $n$ is admissible if for all $i\in\{1,2,\cdots,n-1\}$ we have $A_{\tau_{i-1},\tau_i}=1$.
Let $F$ be a non-empty subset of $E$ and let $F_A^{\infty}:=F^{\mathbb{N}\cup\{0\}}\cap E_A^\infty$.
 We denote by $F^n_A$ $(n\in\mathbb{N})$ the set of all admissible words of length $n$
and by $F^*_A$ the set of all admissible words which have a finite length (i.e. $F^*_A=\cup_{n\in\mathbb{N}}F^n_A$).
For convenience, put $F^0=\lbrace \emptyset \rbrace$.
For $\tau \in F^n_A$ we define the cylinder set of $\tau$ by 
$
[\tau]_F:=\lbrace \tilde\tau \in F_A^\infty:\tilde\tau_i=\tau_i,\  0\leq i\leq n-1\rbrace.$    
If $E=F$ then we will simply write $[\tau]:=[\tau]_E$ for  $\tau\in E^*_A$.
For $n,k\in\mathbb{N}$, $\tilde \tau\in F^n$ and $\tau\in F_A^k$ we set $\tilde{\tau}\tau:=\tilde\tau_0\cdots\tilde\tau_{n-1}\tau_0\tau_1\cdots\tau_{k-1}$ and $\tau'\in \CF$ we set $\tilde\tau\tau':=\tilde\tau_0\cdots\tilde\tau_{n-1}\tau_0'\tau_1'\cdots$. 
Also, for $\tau\in \CF$ and $n\in\mathbb{N}$ we set $\tau|_{n-1}:=\tau_0\cdots \tau_{n-1}$.
We say that the shift space $F_A^\infty$ is finitely primitive if there exist $N_{A,F}\in\mathbb{N}$ and a finite set $\Lambda^F\subset F^{N_{A,F}}_A$ such that for all $e,\tilde e\in F$ there exists $\tau\in \Lambda^F$ such that $e\tau \tilde e$ is admissible. Throughout this paper, we will assume that the shift spaces $\CE$ and $\CF$ are finitely primitive.
We endow $\CF$ with the topology $\mathcal{O}$ generated by the following metric: 
\begin{align*}
     d(\tau,\tau'):= \left\{
 \begin{array}{cc}
   {e^{-k}}   & \text{if}\  \tau_i=\tau_i'\ \text{for\ all}\ i=0,\cdots,k-1\ \text{and}\ \tau_{k}\neq\tau_{k}'\\
   0   & \text{if}\ \tau=\tau'
 \end{array}\ \ (\tau,\tau'\in \CF).
 \right. 
\end{align*} 
 We shall denote by $\mathcal{B}$ the $\sigma$-algebra generated by $\mathcal{O}$.
We define the left-shift $\sigma:\CF\rightarrow \CF$ by
$
\sigma(\tau_0\tau_1\tau_2\cdots):=\tau_1\tau_2\cdots.$

For $F\subset E$ we will introduce spaces of functions on $\CF$.
  We denote by $C(F^\infty_A)$ the set of all continuous  functions on $F^\infty_A$ and by $C^b(F^\infty_A)$ the set of all bounded continuous  functions on $F^\infty_A$. For $g\in C^b(F^\infty_A)$ we define $\|g\|_{\infty}:=\sup\{|g(\tau)|:\tau\in F^\infty_A\}$. Note that the space $C^b(F^\infty_A)$ equipped with the supremum norm $\|\cdot\|_\infty$ is a Banach space. We denote by $(C^b(\CF))^*$ the dual space of $C^b(F^\infty_A)$.
Let $\exponent>0$. For a function $g$ on $\CF$ and $n\in\mathbb{N}$ we define 
\begin{align*}
v_{\exponent,n}(g):=
\sup_{{w}\in E^n_A}
\sup\left\{
\frac{|g(\tau)-g(\tilde \tau)|}{(d(\tau,\tilde\tau))^{\exponent }}
:\tau,\tilde \tau\in [{w}]_F 
\right\}
\text{ and }
v_\exponent(g):=
\sup_{n\in\mathbb{N}}v_{\exponent,n}
\end{align*}
A  function $g$ on $\CF$ is called locally H\"older with exponent $\exponent$ if $v_\exponent(g)<\infty$. We will denote by $H_\exponent(F_A^\infty)$ the set of all locally H\"older functions $g$ on $\CF$  with exponent $\exponent$ and by $H_\exponent^b(F_A^\infty)$ the set of all locally H\"older bounded functions $g$ on $\CF$ with exponent $\exponent$. We endow $H_\exponent^b(F_A^\infty)$ with the H\"older norm $\|\cdot\|_\exponent$ with exponent $\exponent$ defined by $\|g\|_{\exponent}=\|g\|_\infty+v_\exponent(g)$ ($g\in H_\exponent^b(F_A^\infty)$). Then, the space $ H_\exponent^b(F_A^\infty)$ is a Banach space.

Let $(\RS,\mathcal{F},\RP)$ be a complete probability space and let $\RD:\RS\rightarrow\RS$ be an invertible map and bimeasurable. We assume that $\RD:\RS\rightarrow\RS$ is measure preserving and ergodic with respect to the probability space $(\RS,\mathcal{F},\RP)$. We consider the measurable space $(\CF\times\RS,\mathcal{B}\otimes\mathcal{F})$, where $\mathcal{B}\otimes\mathcal{F}$ denotes the product $\sigma$-algebra of $\mathcal{B}$ and $\mathcal{F}$. 
Let  $\varPi_{\RS}:\CF\times\RS\rightarrow\RS$ be the canonical projections onto $\RS$.
Let $\sigma\times\RD:\CF\times\RS\rightarrow\CF\times\RS$ be the skew-product map defined by 
$(\sigma\times\RD)(\tau,\RV):=(\sigma(\tau),\RD(\RV)).$ 
Let $\mathcal{P}_\RP(\CF\times\RS)$ be the space of all random probability measures $ \widetilde m$ on the measurable space $(\CF\times\RS,\mathcal{B}\otimes\mathcal{F})$ whose marginal is $\RP$, i.e. $\widetilde m\circ\varPi_\RS^{-1}=\RP$ and let $\mathcal{I}_{\RP}(\CF\times\RS)$ be the set of all $(\sigma\times \RD)$-invariant measures $\widetilde \mu\in \mathcal{P}_\RP(\CF\times\RS)$. For $\widetilde m\in \mathcal{P}_\RP(\CF\times\RS)$ we denote by $\{\widetilde m_\RV\}_{\RV\in\RS}$ the corresponding the disintegration of $\widetilde m$ (see \cite[Proposition 3.6]{crauel2002random}).
 Since $\RD$ is invertible and $\RP$-preserving, $\widetilde\mu\in\mathcal{I}_{\RP}(\CF\times\RS)$ if and only if $\widetilde\mu_{\RV}\circ \sigma^{-1}=\tilde \mu_{\RD(\RV)}$ for $\RP$-a.e. $\RV\in\RS$.

Next, we describe spaces of function on $\CF\times \RS$.
A  function $f:\CF\times\RS\rightarrow \mathbb{R}$ is called a random continuous function on $\CF$ if $f$ satisfy the following two conditions: (1) for all $\tau\in\CF$ the $\tau$-section $\RV\in\RS \mapsto f_\tau(\RV):=f(\tau,\RV)$ is measurable. (2) for all $\RV\in\RS$ the $\RV$-section $\tau\in \CF\mapsto f_\RV(\tau):=f(\tau,\RV)$ is continuous on $\CF$.
We denote by $C_\RS(\CF)$ the set of all random continuous functions on $\CF$.
    A random continuous function $f\in C_\RS(\CF)$ is bounded if $f$ satisfy the following two conditions: (1) for all $\RV\in\RS$ we have $\|f_\RV\|_{\infty}<\infty$. (2) we have $\|f\|_{\infty}:=\text{ess}\sup\{\|f_\RV\|_\infty:\RV\in\RS\}<\infty$.
We denote by $C_\RS^b(\CF)$ the set of all bounded random continuous functions on $\CF$.
\begin{definition}
    A random continuous function $f\in C_\RS(\CF)$ is random locally H\"older with exponent $\exponent>0$ if we have
    $
    v_\exponent(f_\RV)<\infty$  for all $\RV\in\RS$ and $v_\exponent(f):=\text{ess}\sup\{v_\exponent(f_\RV):\RV\in\RS\}<\infty.$  
\end{definition}
We fix $\exponent>0$.
We denote by $H_{\exponent,\RS}(\CF)$ the set of all locally H\"older functions with exponent $\exponent$.
For $f\in H_{\exponent,\RS}(\CF)$ we set 
\[
\dis:=\bigcap_{k\in\mathbb{Z}}\RD^k(\{\RV\in\RS: v_\exponent(f_{\RV})\leq v_{\exponent}(f)\})
\text{ and }
\bd:=\exp\left( v_\gamma(f)
\sum_{k=0}^\infty
e^{-k\exponent}
\right)
.
\]
Since $\RD$ is invertible and $\RP$-preserving, if $f\in H_{\exponent,\RS}(\CF)$ we have $\RP(\dis)=1$.
We then introduce the random Ruelle operator associated with a  potential $f\in H_{\exponent,\RS}(\CF)$.
Let $f\in H_{\exponent,\RS}(\CF)$.
For $g\in C^b(\CF)$ we define the random Ruelle operator $\TOF$ associated with a potential $f$ by 
\[
\TOF(g)(\tau):=\sum_{e\in F, A_{e,\tau_0}=1}
e^{ f(e\tau,\RV)}g(e\tau).
\]
Let $1_A$ be the indicator function with respect to $A\subset \CF$ and let $1=1_{\CF}$.  
\begin{definition}\label{def summabel}
    A potential $f\in H_{\exponent,\RS}(\CF)$ is called a summable potential if we have 
    \[
    \int \log M_f(\RV) d\RP(\RV)<\infty, \text{ where } 
    M_f(\RV):=\sup\{\TOF({1})(\tau):\tau\in \CF\}.
    \]
Moreover, a potential $f\in H_{\exponent,\RS}(\CF)$ is called a 
normal summable potential if $f$ is summable and we have 
    \[
    \int \log \underline{M}_f(\RV) d\RP(\RV)>-\infty, \text{ where } 
    \underline{M}_f(\RV):=\inf\{\TOF({1})(\tau):\tau\in \CF\}.
    \]
\end{definition}
We denote by $H_{\exponent,\RS}^s(\CF)$ the set of all summable functions and by $H_{\exponent,\RS}^{ns}(\CF)$ the set of all normal summable functions. For $f\in H_{\exponent,\RS}(\CF)$ we set 
\begin{align}\label{eq definision of fin}
\fin:=\bigcap_{k\in\mathbb{Z}}\RD^k(\{\RV\in\RS: M_{f}(\RV)<\infty\}). 
\end{align}
Since $\RD$ is invertible and $\RP$-preserving, if $f\in H_{\exponent,\RS}^s(\CF)$ then we have $\RP(\fin)=1$. 
As in the deterministic case, one can show that for all
$\RV\in\fin$ we have $\TOF(\CFspace)\subset \CFspace$ and for all $\RV\in\dis\cap\fin$ we have $\TOF(\HFspace)\subset \HFspace$. For $n\geq1$ and $\RV\in\fin$ we define 
$\TOF^n(g):=\mathcal{L}_{f,F,\RD^{n-1}(\RV)}\circ\cdots\circ\TOF(g)$.
By a standard calculation, we can show that 
\[
\TOF^n(g)(\tau)=
\sum_{a\in F_A^{n},\ A_{a_{n-1},\tau_0}=1}
\exp\left(
S_nf(a\tau,\RV))g(a\tau
\right),
\]
$\text{where }
S_kf(\tau,\RV):=\sum_{l=0}^{k-1}f((\sigma\times\RD)^l(\tau,\RV)).$
For 
$\RV\in\fin$ we denote by $\TOF^*$ the dual operator of the operator $\TOF:\CFspace\rightarrow\CFspace$. 
For simplicity, if $F=E$ then we will write $\TO:=\mathcal{L}_{f,E,\RV}$,
  $\TO^*:=\mathcal{L}_{f,E,\RV}^*$ and $\TO^n:=\mathcal{L}_{f,E,\RV}^n$ for all $n\in\mathbb{N}$.
For $e\in F$, $\RV\in\RS$ and $n\in\mathbb{N}$ we set
$F_{A,e}^n:=\{\tau\in F^n_A:\tau_0=e,\ A_{\tau_{n-1},e}=1\}$
and define 
\[
Z_{n,e,F}^\RV(f)
:=
\sum_{\tau\in F_{A,e}^n}
\exp
\left(
\sup\{
S_nf(\tilde \tau,\RV)
:
\tilde \tau\in[\tau]_F
\}
\right)
\]
if $F_{A,e}^n\neq\emptyset$ and  $Z_{n,e,F}^\RV(f)=0$ otherwise. For $e\in F$ we fix $\xi_e\in [e]_F$. For $e\in F$, $\RV\in\RS$ and $n\in\mathbb{N}$ we also define 
$
\mathfrak{L}_{n,e,F}^\RV(f):=\TOF^n(1_{[e]_F})(\xi_e).
$
\begin{prop}[{\cite[Proposition 3.1 and Lemma 3.2]{StaOnrandomtop}}]\label{prop definition of the pressure}
    Let $F\subset E$ be a non-empty set and let $f\in H_{\exponent,\RS}^s(\CF)$. We assume that $\CF$ is finitely primitive. Then there exist a measurable set $\RS'\subset \RS$  and a constant $P_F(f)\in \mathbb{R}\cup\{-\infty\}$ such that $\RP(\RS')=1$ and for all $\RV\in\RS'$ and $e\in F$ we have
    \[
    P_F(f)
    =\lim_{n\to\infty}\frac{1}{n}\log Z_{n,e,F}^\RV(f)
    =\lim_{n\to\infty}\frac{1}{n}\log \mathfrak{L}_{n,e,F}^\RV(f)\geq-\infty
    .\]
    Moreover, if $f\in H_{\exponent,\RS}^{ns}(\CF)$ then $P_F(f)$ is finite. 
\end{prop}
For $f\in H_{\exponent,\RS}^s(\CF)$ we call $P_F(f)$ the relative topological pressure with respect to $f$.  We set 
\begin{align}\label{Omega pressure}
\pr:=
\bigcap_{k\in\mathbb{Z}}\RD^k\left(
\left\{
\RV\in\RS:
\lim_{n\to\infty}\frac{1}{n}\log \mathfrak{L}_{n,e,F}^\RV(f)=P_F(f),\ e\in F
\right\}
\right).    
\end{align}
If $E=F$ then we will write $P(f)=P_E(f)$

\begin{thm}{\cite{Denkerkiferstadlbauer, StaOnrandomtop, Stacoupling}}\label{thm perron theorem}
    Let $F$ be a non-empty subset of $E$ such that $\CF$ is finitely primitive and
let $f\in H_{\exponent,\RS}^{ns}(\CF)$. Then there exist a random variable $\lambda^*_f:\RS\rightarrow\mathbb{R}$ with $\int \log \lambda_f^*(\RV) d\RP(\RV)=P_F(f)$, $h\in H_{\exponent,\RS}(\CF)$ and a random probability measure $\left\{\widetilde \nu^f_\RV\right\}_{\RV\in\RS}$ such that the following hold:
\begin{itemize}
    \item[(1)] For $\RP$-a.e. $\RV\in\RS$ the function $h_\RV:\CF\rightarrow \mathbb{R}$ is a bounded strictly positive function satisfying $\TOF h_\RV=\lambda^*_f(\RV) h_{\RD(\RV)}$ and $\int h_{\RV} d\widetilde \nu^f_\RV=1$.
    \item[(2)] For $\RP$-a.e. $\RV\in\RS$ we have 
    \begin{align}\label{eq def conformal measure}
    \TOF^*(\widetilde \nu^f_{\RD(\RV)})=\lambda^*_f(\RV)\widetilde \nu^f_{\RV}.
    \end{align}
In particular, for $\RP$-a.e. $\RV\in\RS$ we have $\lambda^*_f(\RV)=\TOF^*(\widetilde \nu^f_{\RD(\RV)})(1)$.
    \item[(3)] The probability measure $\widetilde \rho_f$ given by $h_\RV d\widetilde \nu^f_{\RV}d\RP$ is in $\mathcal{I}_{\RP}(\CF\times \RS)$.
    \item[(4)] $\left\{\widetilde \nu^f_\RV\right\}_{\RV\in\RS}$ is the unique random probability measure with $\TOF^*(\widetilde \nu^f_{\RD(\RV)})=\lambda^*_f(\RV)\widetilde \nu^f_{\RV}$ for $\RP$-a.e. $\RV\in\RS$. 
\end{itemize}
\end{thm}

For $\widetilde \mu\in \imsF$ we denote by 
$h^r_{\widetilde\mu}(\sigma\times\RD)$ the relative entropy as defined in \cite{kifer2006random}. We have the following variational principle.
\begin{thm}{\cite{Denkerkiferstadlbauer, Stacoupling}}
\label{thm variational principle}
    Let $F\subset E$ be a non-empty set and let $f\in H_{\exponent,\RS}^{ns}(\CF)$. We assume that $\CF$ is finitely primitive. Then, we have
\[
P_F(f)=\sup
\left\{
h_{\widetilde\mu}^r(\sigma\times\RD)+\int fd\widetilde\mu
:
\widetilde\mu\in \imsF \text{ with } \int f d\widetilde\mu>-\infty
\right\}.
\]
Moreover, if $\int f \widetilde \rho_f>-\infty$ then $\widetilde \rho_f$ is
     an equilibrium state for the potential $f$, that is,
    $P_F(f)=h_{\widetilde\rho_f}^r(\sigma\times \RD)+\int f d\widetilde\rho_f.$
\end{thm}

\subsection{Several forms of the relative topological pressure}\label{section Several forms of the relative topological pressure}

Let $F$ be a subset of $E$ such that $\CF$ is finitely primitive. Then there exist $N_{A,F}\geq 1$ and $\Lambda^F\subset F^{N_{A,F}}$ such that $\Lambda^F$ is a finite set which witnesses the finitely primitivity of matrix $A$.
We write $\Lambda^F:=\{w^i\in F^{N_{A,F}}_A:1\leq i\leq \#\Lambda^F\}$, where $\#C$ denote the cardinality of the set $C$. We set 
$\Lambda^F_a:=\{e\in F:$ there exist $1\leq i\leq\#\Lambda^F$ and $0\leq k\leq N_{A,F}-1$ such that $e=w^i_k$\}.
If $F=E$ then we write $N_{A}:=N_{A,E}$, $\Lambda:=\Lambda^F$ and $\Lambda_a:=\Lambda^F_a$
For a potential $f\in H_{\exponent,\RS}(\CF)$, $\RV\in\RS$, $e\in F$ and $n\in\mathbb{N}$ we define 
\begin{align*}
&A_{n,F}^\RV(f):=
\sum_{\tau\in F_A^n}
\exp
\left(
\sup\{S_nf(\widetilde \tau,\RV):\widetilde \tau\in[\tau]_F\}
\right)
\text{ and }
\\&
L_{n,e,F}^\RV(f):=
\sum_{\tau\in F_A^n,\ A_{\tau_{n-1},e}=1}
\exp
\left(
S_nf(\tau\xi_e,\RV)
\right) 
\end{align*}
For simplicity, if $F=E$ then we will write $A_{n}^\RV:=A_{n,E}^\RV$ and $L_{n,e}^\RV:=L_{n,e,E}^\RV$.
For all $n\in\mathbb{Z}$ and $\RV\in\RS$ we set
\[
\RV_n:=\RD^{n}(\RV).
\]
For $F\subset E$, $B\subset \CF$, $n\in\mathbb{N}$, $\RV\in\RS$ and $f\in H_{\exponent,\RS}(\CF)$ we put
$\overline{S_n}f(B,\RV):=\sup\{S_nf(\tau,\RV):\tau\in B\}$ and $\underline{S_n}f(B,\RV):=\inf\{S_nf(\tau,\RV):\tau\in B\}$. For simplicity we write
$\overline{f}(B,\RV):=\overline{S_1}f(B,\RV)$ and $\underline{f}(B,\RV):=\underline{S_1}f(B,\RV)$.

\begin{lemma}\label{lemma comparablity}
Let $F$ be a subset of $E$ such that $\CF$ is finitely primitive and let $f\in H_{\exponent,\RS}(\CF)$.
Then, for all $\RV\in\dis$, $e\in F$ and $n\in\mathbb{N}$ we have
    $\mathfrak{L}_{N_{A,F}+1+n,e,F}^{\RV}(f)
    \geq
R    L_{n,e,F}^{\RV_{N_{A,F}+1}}(f)$
and 
 $   L_{N_{A,F}+n,e,F}^\RV(f)
\geq 
R_n A_{n,F}^\RV(f)$, where\\ 
$
R:=\min_{w\in\Lambda^F}
    e^{
    \underline{S_{N_{A,F}+1}}f([ew]_F,\RV)}$
    and  $R_n:=      \bd^{-1}
           \min_{w\in\Lambda^F}
           e^{
    \underline {S_{N_{A,F}}}f([w]_F,\RV_n)}.$    

\end{lemma}

\begin{proof}
We fix $\RV\in \dis$ and $e\in  F$. Then,  for all $n\in\mathbb{N}$ we have
\begin{align*}
&    \mathfrak{L}_{N_{A,F}+1+n,e,F}^{\RV}(f)
    \geq 
    \sum_{\substack{\widetilde\tau\in F_A^n,\\A_{\tau_{n-1},e}=1}}
    \sum_{\substack{w\in \Lambda^F,\\ew\widetilde\tau\in F^*_A}}
    \exp
    \left(S_{N_{A,F}+1+n}f(ew\widetilde\tau\xi_e),\RV)
    \right)
    \geq
    RL_{n,e,F}^{\RV_{N_{A,F}+1}}(f)
\end{align*}
       Note that for all $n\in\mathbb{N}$, $\tau\in F_A^n$ and $\tilde \tau,\tilde\tau'\in[\tau]_F$ we have
        $   |S_nf(\tilde \tau,\RV)-S_nf(\tilde \tau',\RV)|
           \leq 
           \bd.$
       Therefore, for all $n\in\mathbb{N}$ we obtain
       \begin{align*}
           &L_{N_{A,F}+n,e,F}^\RV(f)
           \geq
           \sum_{\tau\in F_A^n}
           \sum_{\substack{w\in \Lambda^F,\\\tau w\xi_e\in F^*_A}}
           \exp(S_{N_{A,F}+n}f(\tau w\xi_e,\RV))
           \geq
 R_n    A_{n,F}^\RV(f)
     \end{align*}
     \end{proof}

\begin{definition}
    A potential $f\in H_{\exponent,\RS}(\CF)$ is said to be bounded over finite subalphabets if for all finite set $\tilde F\subset F$ we have 
 $   \text{ess}\sup\left\{\left\|\left(f|_{ \tilde F_A^{\infty}\times\RS}\right)_\RV\right\|_{\infty}:\RV\in\RS\right\}<\infty.$
\end{definition}
We will denote by $H_{\exponent,\RS}^{1}(\CF)$ the set of all functions $f\in H_{\exponent,\RS}(\CF)$ whose are bounded over finite subalphabets and put $H_{\exponent,\RS}^{1,s}(\CF)=H_{\exponent,\RS}^{1}(\CF)\cap H_{\exponent,\RS}^{s}(\CF)$. For a potential $f\in H_{\exponent,\RS}(\CF)$ and $\tilde F\subset F$ we set $f_{\tilde F}:=f|_{\tilde F_A^{\infty}\times\RS}$.
For $f\in H_{\exponent,\RS}(\CF)$ 
we define $\K:= \text{ess}\sup\{\|(f_{\Lambda^F_a})_\RV\|_{\infty}:\RV\in\RS\}$ and
\[
\cb:=\bigcap_{k\in\mathbb{Z}}\RD^{k}\left(\left\{\RV\in\RS:\|(f_{\Lambda^F_a})_{\RV}\|_{\infty}\leq\K\right\}\right)
\]
Note that for all $f\in H_{\exponent,\RS}^{1}(\CF)$
we have $\RP(\cb)=1$ 
and for $\RV\in\dis\cap\cb$ we have 
$
\underline{M}_f(\RV)
\geq
\bd^{-1}
e^{-\K}
.$
Therefore, 
we obtain $H_{\exponent,\RS}^{1,s}(\CF)\subset H_{\exponent,\RS}^{ns}(\CF)$.

\begin{lemma}\label{lemma several form of prressure}
 Let $F$ be a subset of $E$ such that $\CF$ is finitely primitive  and let $f\in H_{\exponent,\RS}^{s}(\CF)$. Then for all $e\in F$,
$\RV\in \dis\cap \cb$ and $n\in\mathbb{N}$ we have 
\begin{align*}
&\mathfrak{L}_{n,e,F}^\RV(f)
\leq Z_{n,e,F}^\RV(f)
\leq \bd L_{n,e,F}^\RV(f)
\leq \bd A_{n,F}^\RV(f)
\\&\leq 
\bd^{3}
e^{N_{A,F}\K}
  L_{N_{A,F}+n,e,F}^\RV(f)
\leq 
\bd^{3}
e^{N_{A,F}\K}
R^{-1}
\mathfrak{L}_{2N_{A,F}+1+n,e,F}^{\RV_{-(N_{A,F}+1)}}(f).    
\end{align*}
In particular, for all $\RV\in \pr\cap\dis\cap\cb$ and $e\in F$ we have 
\begin{align*}
   & P_F(f)
   =\lim_{n\to\infty}\frac{1}{n}\log \mathfrak{L}_{n,e,F}^\RV(f)
   = 
        \lim_{n\to\infty}\frac{1}{n}\log Z_{n,e,F}^\RV(f)
    \\&    =\lim_{n\to\infty}\frac{1}{n} \log L_{n,e,F}^\RV(f)
    =\lim_{n\to\infty}\frac{1}{n} \log A_{n,F}^\RV(f).
\end{align*}
\end{lemma}

\begin{proof}
Let $f\in H_{\exponent,\RS}^{s}(\CF)$. We fix $e\in F$.
    By definitions of $Z_{n,e,F}^\RV(f)$, $L_{n,e,F}^\RV(f)$, $\mathfrak{L}_{n,e,F}^\RV(f)$ 
and $A_{n,F}^\RV(f)$, for all $\RV\in \dis$ and $n\in\mathbb{N}$ we have
 $   \mathfrak{L}_{n,e,F}^\RV(f)\leq Z_{n,e,F}^\RV(f)\leq \bd L_{n,e,F}^\RV(f)\leq \bd A_{n,F}^\RV(f).$
By Lemma \ref{lemma comparablity}, for all $\RV\in\dis$ and $n\in\mathbb{N}$ we have 
   $ \mathfrak{L}_{N_{A,F}+1+N_{A,F}+n,e,F}^{\RV_{-(N_{A,F}+1)}}(f)\geq
RL_{N_{A,F}+n,e,F}^\RV(f)$
 and
 $ L_{N_{A,F}+n,e,F}^\RV(f)
\geq 
R_n
A_{n,F}^\RV(f).
$
On the other hand for all $\RV\in \dis\cap\cb$ and $n\in\mathbb{N}$ we have
 \[R_n\geq
    \bd^{-1}\exp\left({-\sum_{j=0}^{N_{A,F}-1}\|(f_{\Lambda^F_a})_{\RD^{j}(\RV_n)}\|_\infty}\right)
\geq \bd^{-1}e^{-N_{A,F}\K}.\]
Thus, the proof is complete.
\end{proof}

Let $F\subset E$ and let $f\in \h^{s}(\CF)$. 
For $n\in\mathbb{N}$, $\RV,\RV'\in\RS$ and a sequence $\{\widetilde{m}_{\RV_n}\}_{n\in\mathbb{N}\cup\{0\}}$ of Borel probability  measures on $\CF$ we define
\begin{align}\label{eq abstruct eigenvalue}
P_{f,\widetilde m_{\RV_n}}(\RV'):=\log\int \mathcal{L}_{f,F,\RV'}(1) d\widetilde{m}_{\RV_n}
\text{ and }
P_{f,\widetilde m}^n(\RV):=\sum_{i=0}^{n-1}P_{f,\widetilde m_{\RV_{i+1}}}(\RV_i).
\end{align}
Note that, by the definitions of $\underline{M}_f(\RV')$ and $M_f(\RV')$ ($\RV'\in\RS$), for all $n\in\mathbb{N}\cup\{0\}$ and $\RV'\in\RS$ we have $\underline{M}_f(\RV')\leq e^{P_{f,\widetilde m_{\RV_n}}(\RV')}\leq M_f(\RV')$.

\begin{lemma}\label{lemma gibbs}
Let $F$ be a subset of $E$ such that $\CF$ is finitely primitive and let $f\in H_{\exponent,\RS}^{s}(\CF)$.
We assume that for $\RV\in\dis\cap\cb$ a sequence $\{\widetilde{m}_{\RV_{n}}\}_{n\in\mathbb{N}\cup\{0\}}$ of Borel probability  measures on $\CF$ satisfies the following condition: For all $n\in\mathbb{N}$ we have $\mathcal{L}^*_{f,F,\RV_{n-1}}(\widetilde m_{\RV_n})=e^{P_{f,\widetilde m_{n}}(\RV_{n-1})}\widetilde m_{\RV_{n-1}}$. Then for all $n\in\mathbb{N}$, $\tau\in F_A^{n}$ and $\tilde \tau\in[\tau]$ we have 
\[
\frac{1}{\bd e^{2N_{A,F}\K}N_{A,F}\prod_{i=0}^{2N_{A,F}-1}M_f(\RV_{n+i})
}\leq \frac{\widetilde m_\RV([\tau])}{\exp(S_nf(\tilde\tau,\RV)-P^n_{f,\widetilde m}(\RV))}\leq \bd.
\]
\end{lemma}
\begin{proof}
     We fix $n\in\mathbb{N}$, $\tau\in E_A^{n}$ and $\tilde \tau\in[\tau]$. By the assumption, we have
    \begin{align*}
        &\widetilde m_\RV([\tau])
        =e^{- P^n_{f,\widetilde m}(\RV)}
        \int 
        \mathcal{L}_{f,F,\RV}^n(1_{[\tau]})
        d\widetilde m_{\RV_{n}}
        \leq \bd e^{S_nf(\tilde\tau,\RV)- P^n_{f,\widetilde m}(\RV)}  
\text{ and }
   \\& \widetilde m_{\RV}([\tau])
         \geq \bd^{-1} e^{S_nf(\tilde \tau,\RV)- P^n_{f,\widetilde m}(\RV)}
        \widetilde m_{\RV_n}({\{\rho\in\CF:A_{\tau_{n-1},\rho_0}=1\}})
\end{align*}
Note that, by the definition of $\Lambda^F\subset F_A^{N_{A,F}}$, we have 
$
\bigcup_{w\in\Lambda^F} \{\rho\in\CF:A_{w_{N_{A,F}-1},\rho_0}=1\}=\CF.
$
Thus, there exists $w\in\Lambda^F$ such that $\widetilde m_{\RV_{n+2N_{A,F}}}(\{\rho\in\CF:A_{w_{N_{A,F}-1},\rho_0}=1\})\geq 1/N_{A,F}$.
Furthermore, there exists $a_{\tau,w}\in \Lambda^F$ such that $\tau a_{\tau,w}w\in F_A^*$.
Therefore, by the assumption in this lemma, we obtain
\begin{align*}
  &  \widetilde m_{\RV_n}({\{\rho\in\CF:A_{\tau_{n-1},\rho_0}=1\}})
\geq 
\widetilde m_{\RV_n}([a_{\tau,w}w])
\\&=e^{-P^{2N_{A,F}}_{f,\widetilde m}(\RV_n)}\int_{\{\rho\in\CF:A_{w_{N_{A,F}-1},\rho_0}=1\}}e^{S_{2N_{A,F}}f(a_{\tau,w}w\rho,\RV_{n})}d\widetilde m_{\RV_{n+2N_{A,F}}}(\rho)
\\&\geq 
\left(
\prod_{i=0}^{2N_{A,F}-1}M(\RV_{n+i})
\right)^{-1}
\bd^{-1}e^{-2N_{A,F}\K}N_{A,F}^{-1}.
\end{align*}
\end{proof}

Let $f\in H_{\exponent,\RS}^{s}(\CF)$. We define 
\begin{align}
\tail:=\bigcap_{k\in\mathbb{Z}}\RD^k\left(\left\{\RV\in\RS:\lim_{n\to\infty}\frac{1}{n}\log M_f(\RV_n)=0\right\}\right). 
\end{align}
Note that if $f\in H_{\exponent,\RS}^{ns}(\CF)$ then $\int |\log M_f(\RV)| d\RP(\RV)<\infty$.
Therefore, by Birkhoff's ergodic theorem, if $f\in H_{\exponent,\RS}^{ns}(\CF)$ then we have $\RP(\tail)=1$. 
\begin{lemma}\label{lemma pointwise pressure and global pressure}
  Let $F$ be a subset of $E$ such that $\CF$ is finitely primitive and let $f\in H_{\exponent,\RS}^{s}(\CF)$.
We assume that a sequence $\{\widetilde{m}_{n}(\RV)\}_{n\in\mathbb{N}\cup\{0\}}$ of Borel probability  measures on $\CF$ satisfy the condition of Lemma \ref{lemma gibbs}.
Then for all $\RV\in\dis\cap\cb\cap\pr\cap\tail$ we have 
    \[
    P(f)=\lim_{n\to\infty}\frac{1}{n}\log A_{n}^\RV(f)=\lim_{n\to\infty}\frac{1}{n}P^n_{f,\widetilde m}(\RV). 
    \]
\end{lemma}
\begin{proof}
    Let $f\in H_{\exponent,\RS}^{s}(\CF)$ and let $\RV\in\dis\cap\cb\cap\pr\cap\tail$. By Lemma \ref{lemma gibbs}, for all $n\in\mathbb{N}$ we obtain
    \[\bd^{-1}
     \leq\frac{A_{n}^\RV(f)}{e^{P^n_{f,\widetilde m}(\RV)}}\leq
     \bd e^{2N_{A,F}\K}N_{A,F}\prod_{i=0}^{2N_{A,F}-1}M(\RV_{n+i})
    \] 
Hence, by Lemma \ref{lemma several form of prressure}, we obtain the desired result.
\end{proof}

\subsection{Compact approximation property of the relative topological pressure}\label{section Compact approximation property of the relative topological pressure}

    For $f\in H_{\exponent,\RS}^1(\CE)\setminus H_{\exponent,\RS}^{s}(\CE)$ (i.e. $\int\log M_fd\RP=\infty$) we set $P(f)=\infty$.

\begin{prop}\label{prop compact apploximation of the pressure}
    Let $f\in H_{\exponent,\RS}^1(\CE)$. We have
    \[
    P(f)=\sup\{P_F(f_F):F\subset E,\ \#F<\infty,\ \CF\text{ is finitely primitive} \}.
    \]
\end{prop}
\begin{proof}
    If $f\in H_{\exponent,\RS}^{1,s}(\CE)$ then the statement follows from \cite[Theorem 4.1]{Denkerkiferstadlbauer}. Let $f\in H_{\exponent,\RS}^1(\CE)\setminus H_{\exponent,\RS}^{s}(\CE)$. 
    Let $\{F_k\}_{k\in\mathbb{N}}\subset E$ be an ascending sequence such that $F_1=\Lambda_a$ and for all $k\in\mathbb{N}$ the set $F_k$ is a finite set, $F_k\subset F_{k+1}$ and $E=\bigcup_{k\in\mathbb{N}}F_k$.
    For all $n,k\in\mathbb{N}$ and $\RV\in\RS$ we will write $f_{k}:=f_{F_k}$, $A_{n,k}^\RV(f_k):=A_{n,F_k}^\RV(f_{F_k})$ and $P_{k}(f_{k}):=P_{F_k}(f_{F_k})$.
    By the monotone convergence theorem, we have
    \begin{align}\label{eq compact approximation lemma average divergence}
        \lim_{k\to\infty}\int \log M_{f_{k}}(\RV)d\RP(\RV)=\infty.
    \end{align}
    Define $\tilde \RD:=\RD^{N_A+1}$.
Since $f\in H^{1}_{\exponent,\RS}(\CE)$ and for each $k\in\mathbb{N}$ the set $F_k$ is finite, for all $k\in\mathbb{R}$ we have $\int |\log M_{f_k}|d\RP<\infty$. Thus, by Birkhoff's ergodic theorem, there exist a measurable function $\log M^*_{f_k}:\RS\rightarrow \mathbb{R}$ ($k\in\mathbb{N}$) and  $\RS_e\subset \RS$ such that $\RP(\RS_e)=1$ and for all $k\in\mathbb{N}$ and $\RV\in \RS_e$ we have       
\begin{align}\label{eq proof compact appro 1}
\log M^*_{f_k}(\RV)=\lim_{n\to\infty}\frac{1}{n}\sum_{i=0}^{n-1}
\log M_{f_{k}}\left(\tilde \RD^i(\RV)\right)    
\end{align}
and $\int \log M^*_{f_k}d\RP=\int\log M_{f_{k}}d\RP$.
We set 
$\RS_{\text{good}}:=\RS_{e}\cap\dis\cap\cb\cap\bigcap_{k\in\mathbb{N}}\RS_{p,f_k}.$
For all $n\in\mathbb{N}$ and $\tau=(\tau^{(1)},\cdots,\tau^{({n})})\in F^{n}$ (not necessarily admissible) there exists $(\alpha^{(1,\tau)},\cdots,\alpha^{(n,\tau)})\in\Lambda^{n}$ such that $\overline{\tau}:=\tau^{(1)}\alpha^{(1,\tau)}\cdots\tau^{(n)}\alpha^{(n,\tau)}\in F_A^{n+N_An}$. 
For all $n,k\in\mathbb{N}$ and $\RV\in\RS_{\text{good}}$ we obtain
\begin{align*}
    & A_{n+N_An,k}^{\RV}(f_k)
      \geq
    \sum_{\tau\in F^n}
    \exp(\underline{S_{n+N_An}}
    f_k([\overline{\tau}]_{F_k},\RV))
    \nonumber
    \\&\geq
    \sum_{\tau\in F^n}
    \prod_{i=0}^{n-1}
    e^{
    \underline{f_k}
    \left([\tau^{(j+1)}]_{F_k},\RV_{i(N_A+1)}\right)}
   C_n
    \geq \bd^{-n}C_n
    \prod_{i=0}^{n-1}A_{1,k}^{\tilde\RD^i(\RV)}(f_k),
\end{align*}
where $C_n:=\prod_{i=0}^{n-1}
    e^{
    \min_{w\in \Lambda}
    \left\{
    \underline{S_{N_A}}
    f_k\left([w]_{F_k},\RV_{i(N_A+1)+1}\right)
    \right\}}$.
On the other hand, for all $\RV\in\RS_{\text{good}}$ and $n,k\in\mathbb{N}$ we have
\begin{align*}
    &\frac{1}{n}\log C_n
\geq
-\frac{1}{n}\sum_{i=0}^{n-1}
\sum_{j=1}^{N_A}
\|(f_1)_{\RD^{i(N_A+1)}(\RV_j)}\|_\infty
-\log B_{f}
\geq-N_A\K-\log \bd
\end{align*}
and, by definitions of $A_{1,k}^{\RV}(f_k)$ and $M_{f_k}(\RV)$,  
we have
$   (1/{n})\log\left(
      \prod_{j=0}^{n-1}A_{1,k}^{\tilde \RD^j(\RV)}(f_k)
    \right)
    \geq (1/{n})\sum_{j=0}^{n-1}\log M_{f_k}(\tilde \RD^j(\RV)).$
Hence, by (\ref{eq proof compact appro 1}), for all $\RV\in\RS_{\text{good}}$ we have
\[
P_k(f_k)\geq\frac{1}{1+N_A}\left(-2\log \bd-N_A\K+\log M_{f_k}^*(\RV)\right).
\]
Since $\RP(\RS_{\text{good}})=1$, the above inequality implies that 
$P_k(f_k)\geq({1+N_A})^{-1}(-2\log \bd
-N_A\K+\int\log M_{f_k}d\RP).$
Hence, by (\ref{eq compact approximation lemma average divergence}), we obtain $\lim_{k\to\infty}P_k(f_k)=\infty$.
\end{proof}

\subsection{Fiberwise multifractal measures}\label{section Multifractal random conformal measures}

Let $F$ be a finite subset of $E$ such that $\CF$ is finitely primitive. Note that if $f\in H_{\exponent,\RS}(\CF)$ then for each $\RV\in\RS$ the potential $\tau\mapsto f_\RV(\tau)=f(\tau,\RV)$ is continuous on the compact set $\CF$. Thus, 
for all $\RV\in\RS$ the operator $\TOF:\CFspace\rightarrow \CFspace$ is bounded. Thus, by \cite[Proposition 3.4]{mayer2011distance}, we obtain the following theorem:
\begin{thm}\label{thm compact version of perron theorem}
    Let $F$ be a finite subset of $E$ such that $\CF$ is finitely primitive. and let $f\in H_{\exponent,\RS}(\CF)$. Then, there exist Borel probability measures $\{\widetilde m_{F,\RV}^f\}_{\RV\in\RS}$ on $\CF$ such that for all $\RV\in\RS$ we have 
    $\TOF^*(\widetilde m_{F,\RD(\RV)}^f)=\lambda_{f,F}(\RV)\widetilde m_{F,\RV}^f, 
    $ where 
    $
    \lambda_{f,F}(\RV)=
    \TOF^*(\widetilde m_{F,\RD(\RV)}^f)(1). 
    $
    \end{thm}
Let $\{F_n\}_{n\in\mathbb{N}}$ be a sequence of finite subsets of $E$ such that $F_1=\Lambda_a$ and for all $n\in\mathbb{N}$ we have $F_n\subset F_{n+1}$ and $E=\bigcup_{n\in\mathbb{N}}F_n$.
For all $f\in H_{\exponent,\RS}^{1,s}(\CE)$, $n\in\mathbb{N}$ and $\RV\in\RS$ we will write $f_n:=f_{F_n}=f|_{(F_n)_A^{\infty}\times \RS}$, $\widetilde m_{n,\RV}:=\widetilde m_{F_n,\RV}^{f_n}$, $\lambda_n(\RV):=\lambda_{f_n,F_n}(\RV)$ and $\mathcal{L}_{f,n,\RV}:=\mathcal{L}_{f_n,F_n,\RV}$.
For all $j,n\in\mathbb{N}$ and $\RV\in \RS$ we set
$
\lambda_{n}^j(\RV):=\prod_{i=0}^{j-1}\lambda_n(\RD^{i}(\RV)).
$
In the rest of this section, we fix the above ascending sequence $\{F_n\}_{n\in\mathbb{N}}$ and $f\in H_{\exponent,\RS}(\CE)$.

\begin{lemma}\label{lemma tight}
    For all $\RV\in\dis\cap\fin$ the sequence $\{\widetilde m_{n,\RV}\}_{n\in\mathbb{N}}$ is tight.
\end{lemma}
\begin{proof}
    Let $\RV\in\dis\cap\fin$. By Lemma \ref{lemma gibbs}, for $n\in\mathbb{N}$, $j\in\mathbb{N}\cup\{0\}$ and a finite set $\widetilde F\subset E$ we have
    \begin{align*}
        &\widetilde m_{n,\RV}
        \left(
        \bigcup_{e\in E\setminus \widetilde F}
        \{\tau\in(F_n)_A^\infty:\tau_j=e\}\right)
        =
        \sum_{e\in E\setminus \widetilde F}
        \sum_{\substack{\tau\in (F_n)_A^{j+1}\\\tau_j=e}}
        \widetilde m_{n,\RV}
        ([\tau]_{F_n})\\&
        \leq \bd\sum_{e\in E\setminus \widetilde F}
        \sum_{\substack{\tau\in (F_n)_A^{j}\\A_{\tau_{j-1},e}=1}}\lambda_n^{-(j+1)}(\RV)
        e^{\overline{S_j}f([\tau],\RV)}
        e^{\overline{f}([e],\RD^{j}(\RV))}
        \leq D_j\sum_{e\in E\setminus \widetilde F}
        e^{\overline{f}([e],\RD^{j}(\RV))}
        \end{align*}
     where $D_{j}:=\bd^{j+1}\left(\prod_{i=0}^{j}\underline{M}_f(\RD^i(\RV))\right)^{-1}
        \prod_{i=0}^{j-1}M_f(\RD^i(\RV))$.
        Note that for all $j\in\mathbb{N}\cup\{0\}$ we have $  \sum_{e\in E}
        e^{\overline{f}([e],\RD^{j}(\RV))}<\infty$. Indeed, for all $j\in\mathbb{N}\cup\{0\}$ we have
        \begin{align}\label{eq convergence of series}
        &\sum_{e\in E}
        e^{\overline{f}([e],\RD^{j}(\RV))}
        \leq
         \bd \#(\Lambda_a)M_f(\RD^j(\RV))<\infty.
        \end{align}
        Thus, for all $\epsilon>0$ and $j\in\mathbb{N}$ there exists a finite set $\widetilde F_j\subset E$ such that we have
$     D_j
        \sum_{e\in E\setminus \widetilde F_j}
        e^{\overline{f}([e],\RD^{j}(\RV))}
        \leq {\epsilon}/{2^{j}}.
$
We fix $\epsilon>0$. Then for all $n\in\mathbb{N}$ we have
$\widetilde m_{n,\RV}\left(E_A^\infty\cap\prod_{j=0}^{\infty} \widetilde F_j\right)
\geq 
1-\epsilon.$
Hence, for all  $\RV\in\dis\cap\fin$ the sequence $\{\widetilde m_{n,\RV}\}_{n\in\mathbb{N}}$ is tight. 
\end{proof}

\begin{thm}\label{thm conformal measure}
    For all $\RV\in\dis\cap\fin$ there exist Borel probability measures $\{\widetilde m_{\RV_{k-1}}^f\}_{k\in\mathbb{N}}$ on $\CE$ such that we have 
    $\TOk^*(\widetilde m_{\RV_k}^f)=\lambda_f(\RV_{k-1})\widetilde m_{\RV_{k-1}}^f$
     where 
    $\lambda_f(\RV_{k-1})=
    \TOk^*(\widetilde m_{\RV_k}^f)(1).$
\end{thm}
\begin{proof}
Let $\RV\in\dis\cap\fin$.
We construct $\{\widetilde m_{\RV_{k-1}}^f\}_{k\in\mathbb{N}}$ inductively. Let $k=1$. 
Since for all $n\in\mathbb{N}$ we have $0<\lambda_{n}(\RV)\leq M_f(\RV)$, there exist a subsequence 
$\{\tilde n_{l}^{1*}\}_{l\in\mathbb{N}}\subset \mathbb{N}$ and $\lambda_f(\RV_0)\in\mathbb{R}$
such that 
\begin{align}\label{eq thm conformal measure lambda one}
    \lim_{l\to\infty}\lambda_{\tilde n_{l}^{1*}}(\RV_0)=\lambda_f(\RV_0).
\end{align}
By Lemma \ref{lemma tight} and Prohorov's theorem, there exists a subsequence 
$\{\tilde n_{l}^{1}\}_{l\in\mathbb{N}}\subset\{\tilde n_{l}^{1*}\}_{l\in\mathbb{N}}$ and a Borel probability measure $\widetilde m_{\RV_0}^f$ on $\CE$ 
such that 
\begin{align}\label{eq thm conformal measure m zero one}
\text{the sequence 
$\{\widetilde m_{\tilde n_{l}^{1},\RV_0}\}_{l\in\mathbb{N}}$
converges weakly to $\widetilde m_{\RV_0}^f$.
}    
\end{align}
Again, by Lemma \ref{lemma tight} and Prohorov's theorem, there exist a subsequence $\{n_{l}^1\}_{l\in\mathbb{N}}\subset\{\tilde n_{l}^{1}\}_{{l\in\mathbb{N}}}\subset\{\tilde n_{l}^{1*}\}_{{l\in\mathbb{N}}}$ and a Borel probability measure $\widetilde m_{\RV_1}^f$ on $\CE$ such that
\begin{align}\label{eq thm conformal measure m one one}
    \text{the sequence 
$\{\widetilde m_{ n_{l}^1,\RV_1}\}_{l\in\mathbb{N}}$ converges weakly to $\widetilde m_{\RV_1}^f$. 
}
\end{align}
We will show that $\mathcal{L}_{f,\RV_{0}}^*(\widetilde m_{\RV_1}^f)=\lambda_f(\RV_0)\widetilde m_{\RV_{0}}^f$. Let $g\in C^b(\CE)$. 
By the triangle inequality and Theorem \ref{thm compact version of perron theorem}, for all $l\in\mathbb{N}$ we have 
\begin{align*}
  &  \left|
    \mathcal{L}_{f,\RV_{0}}^*(\widetilde m_{\RV_1}^f)(g)
    -
    \lambda_f(\RV_0)\widetilde m_{\RV_{0}}^f(g)
    \right|
    \leq 
   \\& 
    \left|
    \mathcal{L}_{f,\RV_{0}}^*(\widetilde m_{\RV_1}^f)(g)
    -
    \mathcal{L}_{f,\RV_{0}}^*(\widetilde m_{n_l^1,\RV_1})(g)
    \right|
    +
    \left|
    \mathcal{L}_{f,\RV_{0}}^*(\widetilde m_{n_l^1,\RV_1})(g)
    -
    \mathcal{L}_{f,n_l^1,\RV_{0}}^*(\widetilde m_{n_l^1,\RV_1})(g)
    \right|
\nonumber\\&+\left|
   \nonumber \lambda_{n_l^1}(\RV_0)\widetilde m_{n_l^1,\RV_0}(g)
    -
    \lambda_{f}(\RV_0)\widetilde m_{n_l^1,\RV_0}(g)
    \right|
+ \left|
    \lambda_{f}(\RV_0)\widetilde m_{n_l^1,\RV_0}(g)
    -
    \lambda_{f}(\RV_0)\widetilde m_{\RV_0}^f(g)
    \right|.
\end{align*}
By (\ref{eq thm conformal measure lambda one}), (\ref{eq thm conformal measure m zero one}) and (\ref{eq thm conformal measure m one one}), the first, third, and fourth terms of the right-hand side in the above inequality converge to zero as $l\to\infty$.
On the other hand, for all $l\in\mathbb{N}$ we have
\begin{align*}
\nonumber&\left|\mathcal{L}_{f,\RV_{0}}^*(\widetilde m_{n_l^1,\RV_1})(g)
    -
    \mathcal{L}_{f,n_l^1,\RV_{0}}^*(\widetilde m_{n_l^1,\RV_1})(g)
    \right|
    \leq \|g\|_{\infty}
    \sum_{e\in E\setminus F_{n_{l}^1}}
          e^{\overline{f}([e],\RV_0)}
          .    
\end{align*}
By (\ref{eq convergence of series}), we obtain
$
\lim_{l\to\infty}
    \sum_{e\in E\setminus F_{n_{l}^1}}
          e^{\overline{f}([e],\RV_0)}=0.$
Hence,  we obtain $\mathcal{L}_{f,\RV_{0}}^*(\widetilde m_{\RV_1}^f)=\lambda_f(\RV_0)\widetilde m_{\RV_{0}}^f$.

Let $k\in\mathbb{N}$. We assume that a measure $\widetilde m_{\RV_{k}}^f$, $\lambda_f(\RV_{k-1})\in\mathbb{R}$ and a subsequence $\{n_{l}^{k}\}\subset\mathbb{N}$
satisfying 
$    \lim_{l\to\infty} \widetilde m_{n_{l}^{k},\RV_{k}}=\widetilde m_{\RV_{k}}^f$ and 
    $\mathcal{L}_{f,\RV_{k-1}}^*(\widetilde m_{\RV_{k}}^f)=\lambda_f(\RV_{k-1})\widetilde m_{\RV_{k-1}}^f$
are defined.
Note that, by the definitions of $\dis$ and $\fin$, for all $j\in\mathbb {N}$ we have $\RV_{j}\in \dis\cap\fin$
Thus, there exists a subsequence 
$\{\tilde n_{l}^{k+1}\}_{l\in\mathbb{N}}\subset \{n_{l}^{k}\}_{l\in\mathbb{N}}$ 
such that 
 $   \lim_{l\to\infty}\lambda_{\tilde n_{l}^{k+1}}(\RV_k)=\lambda_f(\RV_k).$
By Lemma \ref{lemma tight} and Prohorov's theorem, there exist a subsequence $\{n_{l}^{k+1}\}_{l\in\mathbb{N}}\subset\{\tilde n_{l}^{k+1}\}\subset\{ n_{l}^{k}\}$ and a Borel probability measure $\widetilde m_{\RV_{k+1}}^f$ on $\CE$ such that
the sequence 
$\{\widetilde m_{ n_{l}^{k+1},\RV_{k+1}}\}_{l\in\mathbb{N}}$ converges weakly to $\widetilde m_{\RV_{k+1}}^f$. 
By the exactly same argument in the case $k=1$, we can show that $\mathcal{L}_{f,\RV_{k}}^*(\widetilde m_{\RV_{k+1}}^f)=\lambda_f(\RV_k)\widetilde m_{\RV_{k}}^f$. Hence, we are done.
\end{proof}

    For all $\RV\in\dis\cap\fin$ Borel probability measures $\{\widetilde m_{\RV_{k-1}}^f\}_{k\in\mathbb{N}}$ on $\CE$ obtained in Theorem \ref{thm conformal measure} is called fiberwise multifractal measures with respect to $f$ and $\RV$.

\section{Random conformal graph directed Markov systems (RCGDMSs)}\label{section Conformal random graph directed Markov systems (CRGDMSs)}

Let $V$ be a finite set of vertices and let $E$ be a countable  set of directed edges. We define two functions $i:E\rightarrow V$ and $t:E\rightarrow V$ by setting $i(e)$ to be the initial vertex of $e$ and $t(e)$ to be the terminal vertex of $e$. Let $A:E\times E\rightarrow \{0,1\}$ be a incidence matrix  and let $\{X_v\}_{v\in V}$ be a set of non-empty compact subsets $X_v$ $(v\in V)$ of a common Euclidean space $\mathbb{R}^d$ ($d\geq 1$). We set $X:=\bigcup_{v\in V}X_v$ and $t(\tau)=t(\tau_{n-1})$, $i(\tau)=i(\tau_{0})$ for all $n\in \mathbb{N}$ and $\tau\in E^n_A$. For all $\tau\in\CE$ and $n\in\mathbb{N}\cup\{0\}$ we set $\tau|_n:=\tau_0\cdots\tau_{n-1}$.
\begin{definition}\label{def RGDMS}
 Let $(\RS,\mathcal{F},\RP)$ be a complete probability space and let $\RD:\RS\rightarrow \RS$ be an invertible bimeasurable ergodic measure-preserving map with respect to $(\RS,\mathcal{F},\RP)$. $\Phi:=(\RD:\RS\rightarrow\RS, \{\RV\mapsto \phi_{e,\RV}\}_{e\in E})$ is called  a random graph directed Markov system (RGDMS) if for each $\RV\in\RS$ and $e\in E$  the map $\phi_{e,\RV}:X_{t(e)}\rightarrow X_{i(e)}$ is  one-to-one contraction  with a Lipschitz constant at most a common number $0<\cc<1$ and for all $x\in X_{t(e)}$ the map $\RV\mapsto \phi_e(x,\RV):=\phi_{e,\RV}(x)$ is measurable.
\end{definition}
In the remainder of this paper, we fix a complete probability space $(\RS,\mathcal{F},\RP)$ and an invertible bimeasurable ergodic measure-preserving map $\RD:\RS\rightarrow \RS$ with respect to $(\RS,\mathcal{F},\RP)$. Let $\Phi$ be a RGDMS. For $\RV\in \RS$, $n\in\mathbb{N}$ and $\tau \in E^n_A$ we set 
\[
\phi_{\tau,\RV}:=\phi_{\tau_0,\RV}\circ\cdots\circ\phi_{\tau_{n-1},\RV_{k-1}}.
\]

Next, we introduce a random limit set $\{\RV\in\RS\mapsto J_\RV\subset X \}$ generated by a RGDMS $\Phi$. Since the maps $\phi_{e,\RV}$ ($e\in E,\ \RV\in\RS$) have a common contraction ratio $0<\cc<1$, for all $\tau\in\CE$, $\RV\in\RS$ and $n\in\mathbb{N}$ the diameter of non-empty compact set $\phi_{\tau|_{n},\RV}(X_{t(\tau_{n})})$ does not exceed $\cc^n$ and  $\phi_{\tau|_{n},\RV}(X_{t(\tau_{n})})\subset \phi_{\tau|_{n-1},\RV}(X_{t(\tau_{n-1})})$. Therefore, for all $\tau\in \CE$ and $\RV\in\RS$ the set
$\bigcap_{n\in\mathbb{N}}\phi_{\tau|_{n-1},\RV}(X_{t(\tau_{n-1})})$
is singleton. 
Hence, the random coding map $\pi_\RV:\CE\rightarrow X$ given by 
\begin{align}\label{eq def coding map}
    \{\pi_\RV(\tau)\}=\bigcap_{n\in\mathbb{N}}\phi_{\tau|_{n-1},\RV}(X_{t(\tau_{n-1})})
\end{align}
is well-defined.
Note that for each $\RV\in\RS$ the coding map $\pi_\RV$ is a H\"older continuous with the exponent $\exponent=-\log \cc$. Furthermore, for all $\tau\in \CE$ the map $\RV\in\RS\mapsto \pi_\RV(\tau)$ is measurable (see \cite[P.15]{RGDMS}).
We define a random limit set $\{\RV\in\RS\mapsto J_\RV\subset X\}$ by setting
\begin{align}\label{eq def limit set }
    J_\RV:=\pi_\RV(\CE) \ \ (\RV\in\RS).
\end{align}
Let $F\subset E$. We define $\Phi_F:=(\RD:\RS\rightarrow \RS,\ \{\RV\mapsto \phi_{e,\RV}\}_{e\in F})$. We denote by $\pi_{\RV,F}$ ($\RV\in\RS$) the coding map with respect to $\Phi_F$. We define a random limit set $\{\RV\in\RS\mapsto J_{\RV,F}\subset X\}$ generated by $\Phi_F$ by setting 
$
J_{\RV,F}:=\pi_{\RV,F}(\CF).$
\begin{definition}\label{def CRGDMS}
    A RGDMS $\Phi$ is said to be conformal (RCGDMS) if the following properties are satisfied:
    \begin{itemize}
        \item[$(C1)$] For each $v\in V$ the set $X_v$ is a compact connected subset of $\mathbb{R}^d$ which is the closure of its interior, that is,  $X_v=\overline{\text{Int}(X_v)}$.
        \item[$(C2)$] (Open set condition) For $\RP$-a.e. $\RV\in\RS$ and all $e,e'\in E$ with $e\neq e'$ we have
        $\phi_{e,\RV}(\text{Int}(X_{t(e)}))\cap\phi_{e',\RV}(\text{Int}(X_{t(e')}))=\emptyset$
        \item[$(C3)$]
        For every vertex $v\in V$ there exists a bounded open connected set $W_v$ such that $X_v\subset W_v\subset \mathbb{R}^d$ and for all $\RV\in\RS$ and each $e\in E$ with $t(e)=v$ the map $\phi_{e,\RV}$ extends to a $C^1$ conformal diffeomorphism from $W_v$ onto its image. Moreover for all $e\in E$ and $x\in W_{t(e)}$ the map $\RV\in\RS\mapsto\phi_{e,\RV}(x)$ is measurable.
        \item[$(C4)$] (Cone condition) There exist $\mathfrak{a}\in (0,\pi/2)$ and $\boldsymbol{u}\in\mathbb{R}^d$ such that for each $v\in V$ and $x\in X_v$ we have $\text{Con}(x,\mathfrak{a},\boldsymbol{u})\subset \text{Int}(X_v)$, where
        $\text{Con}(x,\mathfrak{a},\boldsymbol{u}):=\{y\in \mathbb{R}^d:\langle y-x,\boldsymbol{u}\rangle>|y-x||\boldsymbol{u}|\cos\mathfrak{a} \text{ and } |y-x|<\boldsymbol{u}\}$ and $\langle\cdot,\cdot\rangle$ denotes the usual inner product in $\mathbb{R}^d$ and $|\cdot|$ denotes the Euclidean norm on $\mathbb{R}^d$.
        \item[$(C5)$] 
        There exists $\diR'\subset \RS$ such that $\RP(\diR')=1$ and there exist two constants $L\geq 1$ and $0<\alpha_{\Phi}<1$ such that for all $\RV\in\diR'$, $e\in E$ and $x,y\in W_{t(e)}$ we have 
        $\left|
        |\phi_{e,\RV}'(y)|-|\phi_{e,\RV}'(x)|
        \right|
        \leq 
        L
        |\phi_{e,\RV}'(x)|
        |y-x|^{\alpha_\Phi}.
        $
    \end{itemize}
\end{definition}
    As a deterministic case, we can show the following lemma (see \cite[Lemma 19.3.4]{urbanski2022non})
    \begin{lemma}\label{lemma Holder continuity of the geometry potential}
        Let $\Phi$ be a RCGDMS. Then for all $\RV\in \bigcap_{k\in\mathbb{Z}}\RD^k\left(\diR'\right)$, $\tau\in E_A^*$ and $x,y\in W_{t(\tau)}$ we have
    \[
    \left|
    \log|\phi_{\tau,\RV}'(y)|-\log|\phi_{\tau,\RV}'(x)|
    \right|
    \leq \frac{L}{1-\cc^{\alpha_\Phi}}|y-x|^{\alpha_\Phi}.
    \]
    \end{lemma}
For $\tau\in E_A^*$ and $\RV\in\RS$ we define
$\|\phi_{\tau,\RV}'\|_{\infty}:=\sup\{|\phi'_{\tau,\RV}(x)|:x\in W_{t(\tau)}\}$.
    As an immediate consequence of Lemma \ref{lemma Holder continuity of the geometry potential}, we obtain the bounded distortion lemma:
    \begin{lemma}\label{lemma bounded distortion}
        Let $\Phi$ be a RCGDMS. There exists a constant $K_{bd}\geq 1$ such that for all $\RV\in\diR$, $\tau\in E_{A}^*$ and $x,y\in W_{t(\tau)}$ we have
        \[
        \frac{1}{K_{bd}}\leq 
        \frac{|\phi_{\tau,\RV}'(x)|}{|\phi_{\tau,\RV}'(y)|}\leq K_{bd}
        \]
        Moreover, for all $\RV\in\diR$, $n,k\in\mathbb{N}$, $\tau\in E_A^n$ and
        $a\in E_{A}^k$ with $\tau a\in E_A^*$ we have 
        $
        K^{-1}_{bd}\|\phi'_{\tau,\RV}\|_{\infty}
        \|\phi'_{a,\RV_n}\|_{\infty}
        \leq \|\phi'_{\tau a}\|_{\infty}
        \leq
        \|\phi'_{\tau,\RV}\|_{\infty}
        \|\phi'_{a,\RV_n}\|_{\infty}
        $
    \end{lemma}
For $A,B\subset \mathbb{R}^d$ we define $\text{diam}(A):=\sup\{|x-y|: x,y\in A\}$ and $\text{dist}(A,B):=\inf\{|x-y|:{x\in A},\ y\in B\}$. If $A=\{x\}$ ($x\in\mathbb{R}^d$) then for $B\subset \mathbb{R}^d$ we will write $\text{dist}(x,B):=\text{dist}(\{x\},B)$.
We denote by  $B(x,r)$ the Euclidean open ball centered at $x\in\mathbb{R}^d$ having the radius $r>0$ and by $\partial A$ the Euclidean boundary of the set $A\subset \mathbb{R}^d$.
 The following lemmas can be shown as in the argument \cite[pp. 746-747]{urbanski2022non}
    \begin{lemma}\label{lemma diameter condition}
    Let $\Phi$ be a RCGDMS.
        There exists a constant $D_{\Phi}\geq 1$ such that for all $\RV\in\diR$ and $\tau\in E_A^*$ we have
  $D^{-1}_{\Phi}\|\phi_{\tau,\RV}\|_\infty\leq 
  \text{diam}(\phi_{\tau,\RV}(X_{t(\tau)}))\leq D_{\Phi}\|\phi_{\tau,\RV}\|_\infty.$
     \end{lemma}
\begin{lemma}\label{lemma boll conparability}
    Let $\Phi$ be a RCGDMS. For all $\RV\in\diR$, $\tau\in E^*_A$, $x\in W_{t(\tau)}$ and $0<r\leq\text{dist}(x,\partial W_{t(\tau)})$ we have
    $\phi_{\tau,\RV}(B(x,r))\subset B(\phi_{\tau,\RV}(x),\|\phi_{\tau,\RV}'\|_{\infty}r)$
     and 
    $B(\phi_{\tau,\RV}(x),K^{-1}_{bd}\|\phi_{\tau,\RV}'\|_{\infty}r)
    \subset
    \phi_{\tau,\RV}(B(x,r))$
\end{lemma}

Let $\Phi$ be a RCGDMS and let $F\subset E$. We define the geometric potential  $\zeta_F:\CF\times\RS\rightarrow \mathbb{R}$  by
\begin{align}\label{eq def geometric potential}
\zeta_F(\tau,\RV):=\log|\phi_{\tau_0,\RV}'
(\pi_{\RD(\RV)}(\sigma(\tau)))|.    
\end{align}
If $F=E$ then we will write $\zeta=\zeta_E$. 
Note that for all $\RV\in\RS$ the map $(\zeta_F)_\RV$ is continuous and for all $\tau\in \CF$ the map $\RV\mapsto \zeta_F(\tau,\RV)$ is measurable. Thus, $\zeta_F\in C_\RS(\CF)$. Moreover, by Lemma \ref{lemma Holder continuity of the geometry potential}, $\zeta_F$ is random locally H\"older with exponent $-\alpha_{\Phi}\log \cc$ and thus,  
\[
\RP(\diR)=1,
\text{ where } \diR:=\RS_{d,\zeta}.
\]
In the rest of this paper, we fix $\exponent:=-\alpha_{\Phi}\log \cc$.

\subsection{Bowen's formula}
Let $\Phi$ be a normal RCGDMS. 
Then for all $s\in\mathbb{R}$ we have $s\zeta\in H_{\exponent,\RS}^1(\CE)$. Therefore, by Proposition \ref{prop compact apploximation of the pressure}, we obtain the following:
\begin{prop}\label{prop compact approximation geometric pressure}
    Let $\Phi$ be a normal RCGDMS. Then for all $s\in\mathbb{R}$ we have
    \[
    P(s\zeta)=\{P_F(s\zeta_F):F\subset E,\ \#F<\infty,\ \CF\text{ is finitely primitive}\}.
    \]
\end{prop}
By using Theorem \ref{thm perron theorem} and Proposition \ref{prop compact approximation geometric pressure}, we can apply arguments in \cite[Section 3]{RGDMS} for the relative topological pressure $P(s\zeta)$. In particular, we obtain the following result.
\begin{thm}\label{Thm Bowen formula}
    Let $\Phi$ be a normal RCGDMS. Then for $\RP$-a.e. $\RV\in\RS$ we have
    \begin{align}\label{eq genelarized Bowen formula}
    \dim_H(J_\RV)=\inf \{s\geq 0: P(s\zeta)\leq 0\}.    
    \end{align}
    \end{thm}
To compare Theorem \ref{Thm Bowen formula} and Bowen's formula established by Roy and Urba\'nski \cite{RGDMS}, we recall the some definitions used in \cite{RGDMS}.
For a RCGDMS $\Phi$ we define 
\[
\RU:=
\left\{s\in\mathbb{R}:M_{\text{RU}}(s)<\infty 
\right\},
\text{ where }
M_{\text{RU}}(s)=
\sum_{e\in E}\text{ess}\sup
\left\{\|\phi_{e,\RV}'\|_\infty^s:\RV\in\RS\right\}
.
\]
In \cite{RGDMS} for a RCGDMS $\Phi$ the relative topological pressure is defined as follows:
\begin{align}\label{def RU pressure}
    P_{\text{RU}}(s\zeta):=
    \left\{
 \begin{array}{cc}
   P(s\zeta)   & \text{if}\  t\in \RU\\
   \infty   & \text{otherwise } 
 \end{array}
 \right. .
\end{align}
Note that for a RCGDMS $\Phi$ we have $ Fin_{\text{RU}}\subset  Fin$ and if $E$ is a finite set then we have $ Fin_{\text{RU}}= Fin=\mathbb{R}$.
Thus, if $E$ is finite set then for all $s\in\mathbb{R}$ we have
$P_{\text{RU}}(s\zeta)=P(s\zeta)$.
\begin{definition}\label{def evenly varing}  
A RCGDMS is said to be evenly varying if 
\[
\sup_{e\in E}\frac{\text{ess} \sup\{\|\phi_{e,\RV}'\|_\infty:\RV\in\RS\}}{\text{ess} \inf\{\|\phi_{e,\RV}'\|_\infty:\RV\in\RS\}}<\infty.
\]
\end{definition}
\begin{rem}\label{rem evenly varying}
    In \cite{RGDMS}, for evenly varying RCGDMS Roy and Urba\'nski show the following compact approximation property for the relative topological pressure $P_{\text{RU}}(t\zeta)$ (see \cite[Lemma 3.17]{RGDMS}): For all $s\in\mathbb{R}$ we have
\begin{align}\label{eq RU compact apploximation}
    P_{\text{RU}}(s\zeta)=\{P_F(s\zeta_F):F\subset E,\ \#F<\infty,\ \CF\text{ is finitely primitive}\}.
\end{align}
    This property of the relative topological pressure $P_{\text{RU}}(t\zeta)$ played a rule in the proof of Bowen's formula \cite[Theorem 3.26]{RGDMS}.
\end{rem}
The following Bowen's formula is obtained by \cite[Theorem 3.26]{RGDMS}: For a normal evenly varying RCGDMS $\Phi$ and $\RP$-a.e. $\RV\in\RS$ we have
\begin{align}\label{eq bowen formula}
\dim_H(J_\RV)=\inf \{s\geq 0: P_{\text{RU}}(s\zeta)\leq 0\}.    
\end{align}

The following example show that there is a normal random conformal graph directed Markov system that does not satisfy the evenly varying assumption, for which Bowen's formula stated in Theorem \ref{Thm Bowen formula} holds, while above Bowen's formula fails to hold.
\begin{example}\label{example Bowen formula}
    Let $\RS=\mathbb{N}^{\mathbb{Z}}$ and let $\RP$ be the Bernoulli measure such that for all $i\in\mathbb{N}$ we have  
    \[
    \RP([i]_\RS)=
    C^{-1}_\RS
    \left(
    2^i\sum_{k=1}^{i}2^{k^2}
    \right)^{-1},
    \text{ where }
    C_\RS=\sum_{i=1}^{\infty}
    \left(
    2^i\sum_{k=1}^{i}2^{k^2}
    \right)^{-1}
    \]
and $[i]_{\RS}=\{\RV\in\RS=\mathbb{N}^\mathbb{Z}:\RV_0=i\}$.
Let $\RD:\mathbb{N}^{\mathbb{Z}}\rightarrow \mathbb{N}^{\mathbb{Z}}$ be the left-shift map. Let $E=\mathbb{N}$.
For $i\in\mathbb{N}$ and $\RV\in[i]_\RS$ we define a system $\{\phi_{e,\RV}:[0,1]\rightarrow [0,1]\}_{e\in E}$ consisting similarity maps $\phi_{e,\RV}$ such that  we have
\begin{align*}
    \|\phi_{e,\RV}'\|_\infty=
    \left\{
 \begin{array}{cc}
   2^{-2}   & \text{if}\ e=1\\
   2^{-(l^2+l)}   & \text{if}\ \sum_{k=1}^{l-1}2^{k^2-1}< e\leq \sum_{k=1}^{l}2^{k^2-1} \text{ and } 2\leq l\leq i\\
   8^{-e}   & \text{otherwise } 
 \end{array}
 \right.     
\end{align*}
and the intervals $\{\phi_{e,\RV}([0,1])\}_{e\in E}$ are mutually disjoint.  Note that since for all  $i\in\mathbb{N}$ and $\RV\in[i]_\RS$ we have
\[
\sum_{e\in E}\|\phi_{e,\RV}'\|_\infty\leq \sum_{l=1}^{\infty}\sum_{k=1}^{2^{l^2-1}}2^{-(l^2+l)}=\frac{1}{2},
\]
the last requirement can be satisfied. It is easy to see that $\Phi:=(\RD:\RS\rightarrow\RS,\{\RV\mapsto\phi_{e,\RV}\}_{e\in E})$ is a normal RCGDMS and not evenly varying. By the definition of $\Phi$, for all $0<s\leq 1$ we obtain
\begin{align}\label{eq example our bowen formula RU}
    M_{\text{RU}}(s)=\sum_{l=1}^{\infty}\sum_{k=1}^{2^{l^2-1}}2^{-s(l^2+l)}=
    \left\{
 \begin{array}{cc}
    2^{-1}  & \text{if}\ s=1\\
   \infty & \text{otherwise } 
 \end{array}
 \right..     
\end{align}
On the other hand, for all $i\in\mathbb{N}$, $\RV\in[i]_\RS$ and $0<s\leq 1$ we have
\begin{align*}
    M_{s\zeta}(\RV)=\sum_{k=1}^{l_{i}}\|\phi_{k,\RV}'\|^s
    +\sum_{k=l_i+1}^{\infty}\frac{1}{8^{sk}}
    \leq \sum_{l=1}^{i}2^{l^2(1-s)-ls-1}
    +\frac{1}{8^s-1}\leq \sum_{l=1}^{i}2^{l^2}
    +\frac{1}{8^s-1},
\end{align*}
where $l_i:=\sum_{k=1}^i2^{k^2-1}$.
Thus, for all $0<s\leq 1$ we obtain
\begin{align}\label{eq example our Bowen formula dks}
    \int \log M_{s\zeta}(\RV) d\RP(\RV)\leq C_\RS^{-1} \sum_{i=1}^{\infty}\frac{\log\left(\sum_{l=1}^{i}2^{l^2}
    +\frac{1}{8^s-1} \right) }{2^i\sum_{l=1}^{i}2^{l^2}}<\infty.
\end{align}
By Proposition \ref{prop definition of the pressure} and the definition of $M_{\text{RU}}(s)$, for all $0<s\leq 1$ we have $P(s\zeta)\leq \int \log M_{s\zeta}(\RV)d\RP(\RV)\leq \log M_{\text{RU}}(s)$.
In particular, by (\ref{eq example our bowen formula RU}), we have $P(\zeta)\leq -\log 2<0$.
Therefore, by (\ref{eq example our Bowen formula dks}) and Theorem \ref{Thm Bowen formula}, for $\RP$-a.e. $\RV\in\RS$ we have $\dim_H(J_\RV)=\inf \{s\geq 0: P(s\zeta)\leq 0\}<1.$ Hence, by (\ref{eq example our bowen formula RU}), for $\RP$-a.e. $\RV\in\RS$ we obtain
\[
\dim_H(J_\RV)<1= \{s\geq 0: P_{\text{RU}}(s\zeta)\leq 0\}.
\]
This implies that Bowen's formula with respect to $P(s\zeta)$  hold, while  Bowen's formula with respect to $P_{\text{RU}}(s\zeta)$ fails to hold. We would stress that for $\Phi$ and all $0<s<1$ the equation (\ref{eq RU compact apploximation}) is not valid. Indeed, by (\ref{eq example our Bowen formula dks}) for all $0<s<1$ we have 
\begin{align*}
    &\{P_F(s\zeta_F):F\subset E,\ \#F<\infty,\ \CF\text{ is finitely primitive}\}
    \\&\leq
    \int \log M_{s\zeta}(\RV) d\RP
    <
    \infty=P_{\text{RU}}(s\zeta).
\end{align*}
\end{example}

    

\section{Multifractal analysis of the Lyapunov exponent for 
RCGDMSs}\label{section multifractal}
Throughout this section, for a RCGDMS $\Phi$ we always assume that  $\#E=\infty$.
\begin{definition}\label{def RBSCF}
    A RCGDMS $\Phi$ is said to satisfy the random boundary separation condition with over finite subedges (RBSC) if for all finite set $F\subset E$ we have
    \[
    \OSC:=
    \min_{v\in V} \inf\left\{
    \text{dist} \left(\partial X_{v},\bigcup_{\substack{e\in F,i(e)=v}}\phi_{e,\RV}(X_{t(e)})\right):\RV\in \RS
    \right\}>0.
    \]
\end{definition}
It is easy to see that if $\Phi$ satisfy RBSC then for all $\RV\in \RS$ 
the coding map $\pi_\RV:\CE\rightarrow J_\RV$ is bijection.

 Let $\Phi$ be a RCGDMS satisfying RBSC. 
For $\RV\in\RS$, a Borel probability measure $\nu_\RV$ on $J_{\RV,F}$ and $\alpha\in \mathbb{R}$ we define level sets by 
\begin{align*}
&K_{\RV,\nu_\RV}(\alpha):=
\left\{
x\in J_{\RV,F}: \lim_{r\to +0}\frac{\log\nu_\RV(B(x,r))}{\log r}=\alpha 
\right\}\text{ and }\\&
{K}_{\RV,\nu_\RV}^{M}(\alpha):=
\left\{
\pi_{\RV}(\tau)\in J_{\RV,F}:
\lim_{n\to\infty}
\frac{\log\nu_\RV\left(\phi_{\tau|_n,\RV}\left(X_{t\left(\tau|_n\right)}\right)\right)}{\log\left|\phi_{\tau|_n,\RV}'(\pi_{\RV_{n+1}}(\sigma^{n+1}(\tau)))\right|}=\alpha
\right\}.    
\end{align*}
The local dimension spectrum is the family of the functions $\{g_{\RV,\nu_\RV}:\mathbb{R}\rightarrow\mathbb{R}\}_{\RV\in\RS}$ and  given by
$g_{\RV,\nu_\RV}(\alpha):=\dim_H(K_{\RV,\nu_\RV}(\alpha)).$
We also define the Markov dimension spectrum $\{g_{\RV,\nu_\RV}^M:\mathbb{R}\rightarrow \mathbb{R}\}_{\RV\in\RS}$ by 
$ 
g_{\RV,\nu_\RV}^M(\alpha):=\dim_H(K_{\RV,\nu_\RV}^M(\alpha)).$

For a RCGDMS $\Phi$ we set 
\[
\lowb:=\bigcap_{k\in\mathbb{Z}}\RD^{k}(\{\RV\in\RS: \inf\{|\phi_{e,\RV}'(x)|:x\in X_{t(e)}\}\geq M_e\text{ for all } e\in E\}).
\]
Since $\RD$ is invertible and $\RP$-preserving, we have $\RP(\lowb)=1$ if $\Phi$ is normal.
Moreover, if $\Phi$ is normal then $\zeta\in H_{\exponent,\RS}^{1}(\CE)$.

\begin{prop}\label{prop markov dimension and local dimension}
Let $\Phi$ be a RCGDMS $\Phi$ satisfying RBSC and let $F$ be a finite set.
For  all  $\RV\in \diR\cap \lowb
$, a Borel probability measure $\nu_{\RV}$ on $J_{\RV,F}$ and $\alpha\in \mathbb{R}$ we have
$
K_{\RV,\nu_\RV}(\alpha)
=
{K}_{\RV,\nu_\RV}^M(\alpha).
$
\end{prop}
\begin{proof}
    Let $\RV\in \diR\cap \lowb
    $
    and let $\alpha\in \mathbb{R}$. We first show ${K}_{\RV,\nu_\RV}^M(\alpha)\subset K_{\RV,\nu_\RV}(\alpha)$. Let $x\in {K}_{\RV,\nu_\RV}^M(\alpha)$.
    Then there exists the unique $\tau\in \CF$ such that $\pi_{\RV}(\tau)=x$.
    Let $0<r<\OSC D_{\Phi} \|\phi_{\tau_0,\RV}\|_\infty$.  We will denote by $j(x,r)$ the unique natural number satisfying 
    \begin{align*}
       & \frac{\OSC}{K_{bd}}
    \left|
    \phi_{\tau|_{j(x,r)+1},\RV}'
    \left(\pi_{\RV_{j(x,r)+2}}
    \left(\sigma^{j(x,r)+2}
    \left(\tau
    \right)
    \right)
    \right)\right|
    \leq 
    r
    \\&<
    \frac{\OSC}{K_{bd}}
\left|\phi_{\tau|_{j(x,r)},\RV}'
    \left(\pi_{\RV_{j(x,r)+1}}
    \left(\sigma^{j(x,r)+1}
    \left(\tau
    \right)
    \right)
    \right)
    \right|.
    \end{align*}
    Note that, since 
    $\Phi$ has RBSC, we have
    $\text{dist}\left(\pi_{\RV_{j(x,r)+1}}\left(\sigma^{j(x,r)+1}(\tau)\right), \partial W_{t\left(\tau_{j(x,r)}\right)}\right)\geq \OSC.$    
    Therefore, by Lemma \ref{lemma boll conparability} and the definition of $j(x,r)$, we obtain
    \begin{align}\label{eq proof of pesin boll zero}
        &B(x,r)
        \subset 
        \phi_{\tau|_{j(x,r)},\RV}
        \left(
        B\left(
        \pi_{\RV_{j(x,r)+1}}\left(\sigma^{j(x,r)+1}(\tau)\right)
        ,\OSC
        \right)
        \right).
    \end{align}
On the other hand, there exists $l(x,r)\in\mathbb{N}$ such that 
$
D_{\Phi}\|\phi_{\tau|_{l(x,r)+1},\RV}'\|_\infty
< r\leq
D_{\Phi}\|\phi_{\tau|_{l(x,r)},\RV}'\|_\infty.
$
By Lemma \ref{lemma diameter condition}, we have 
\begin{align}\label{eq proof of markov and local lower}
        \phi_{\tau|_{l(x,r)+1},\RV}
\left(
X_{t\left(\tau|_{l(x,r)}\right)}\right)\subset B(x,r).
\end{align}
Since $F$ is a finite set, we have $\min\{M_e:e\in F\}>0$. Hence, for all $\RV'\in\lowb$, $e\in F$ and $x\in X_{t(e)}$ we have $|\phi'_{e,\RV'}(x)|\geq \min\{M_e:e\in F\}>0$. 
Therefore, by (\ref{eq proof of pesin boll zero}) and (\ref{eq proof of markov and local lower}), we obtain $x\in K_{\RV,\nu_\RV}(\alpha)$ and thus, ${K}_{\RV,\nu_\RV}^M(\alpha)\subset K_{\RV,\nu_\RV}(\alpha)$. By similar argument, we can also show that $K_{\RV,\nu_\RV}(\alpha)\subset  K_{\RV,\nu_\RV}^M(\alpha)$.
\end{proof}

For a RCGDMS $\Phi$ and a subset $F$ of $E$ such that $\CF$ is finitely primitive we define ${Fin}_F:=\{s\in\mathbb{R}:s\zeta\text{ is summable}\}
$,
\[
s_{\infty,F}:=\inf\{s\in\mathbb{R}:s\zeta\text{ is summable}\}
\text{ and }
Fin^\circ_F:=(s_{\infty,F},\infty)
\]
Note that we have $Fin^\circ_F= Fin_F$ or $ Fin_F\setminus Fin^\circ_F=\{s_\infty\}$ and if $E=F$ then we have $ Fin= Fin_E$, $s_{\infty}=s_{\infty,E}$ and $Fin^\circ=Fin^\circ_E$. We define the following set (see (\ref{eq definision of fin}) for the definition of $\RS_{b,s\zeta_F}$ for $s\in  Fin_F$):
\begin{align*}
\finFPhi:=
    \left\{
 \begin{array}{cc}
  \bigcap_{n\in\mathbb{N}}\RS_{b,(s_{\infty,F}+1/n)\zeta_F}  & \text{if}\ s_{\infty,F}\notin Fin_F\cup\{-\infty\}\\
   \RS_{b,s_{\infty,F}\zeta_F}   & \text{if}\ s_{\infty,F}\in Fin_F\\
  \bigcap_{n\in\mathbb{N}}\RS_{b,-n\zeta_F}   & \text{if }s_{\infty,F}=-\infty 
 \end{array}
 \right.     
\end{align*}
Since for all $\tau\in \CF$ and $\RV\in\RS$ we have $\zeta_F(\tau,\RV)<0$, for all $s<s'$ and $\RV\in\RS$ we have
\begin{align}\label{eq monotonicity}
M_{s\zeta_F}(\RV)\geq M_{s'\zeta_F}(\RV) \text{ and } \underline{M}_{s\zeta_F}(\RV)\geq \underline{M}_{s'\zeta_F}(\RV). 
\end{align}
Therefore, for all $\RV\in \finFPhi\cap \lowb$ and $s\in Fin$ we have $0<\underline{M}_{s\zeta_F}(\RV)\leq M_{s\zeta_F}(\RV)<\infty$.  
We define the function 
$p_F:\mathbb{R}\rightarrow\mathbb{R}$ by 
\begin{align*}
    p_F(s):
    =\left\{
 \begin{array}{cc}
   {P_F(s\zeta_F)}   & \text{if}\  s\in{Fin}_F\\
   \infty   & \text{otherwise } 
 \end{array}
 \right. .
\end{align*}
If $E=F$ then we have $p=p_F$.

\begin{lemma}\label{lemma convexity of the toporogical pressure}
Let $\Phi$ be a normal RCGDMS and let $F$ be a subset of $E$ such that $\CF$ is finitely primitive. Then the  function $p_F$ is convex on $Fin$. Moreover,
there exists $\preintervalF'\subset \RS$ such that $\RP\left(\preintervalF'\right)=1$, 
    and for all  $\RV\in\preintervalF'$, $s\in  Fin_F$ and $e\in F$ we have
    \[
    P(s\zeta_F)=\lim_{n\to\infty}\frac{1}{n}
    \log \mathfrak{L}_{n,e,F}^\RV(s\zeta_F).
    \]
\end{lemma}
\begin{proof}
The convexity of the function $p_F$ on $Fin$ follows from H\"older's inequality. We will show the second half of the claim in this lemma. For $k\in\mathbb{N}$ we set $C_k:=(s_\infty-1/k, k)$ if $s_\infty\in \mathbb{R}$ and $C_k:=(-k,k)$ otherwise. For all $k\in\mathbb{N}$ we set $C_k':=C_k\cap\mathbb{Q}$. Since $C_k'$, $F$ and $\mathbb{N}$ are countable sets, for all $k\in\mathbb{N}$ 
     there exists $\RS_k\subset \finFPhi\cap \lowb$ such that $\RP(\RS_k)=1$ 
    and for all $\RV\in \RS_{k}$, $s\in C_k'$ and $e\in F$ we have
$    p_F(s)=\lim_{n\to\infty}({1}/{n})\log\mathfrak{L}_{n,e,F}^\RV(s\zeta_F)
    .$
    We fix $k\in\mathbb{N}$, $\RV\in \RS_k$ and $e\in F$.
    Note that for all $n\in\mathbb{N}$ and $s\in Fin$ we have
    $0<\prod_{i=0}^{n-1} \underline{M}_{s\zeta_F}(\RD^{i}(\RV))\leq\mathfrak{L}_{n,e,F}^\RV(s\zeta_F)\leq \prod_{i=0}^{n-1} {M}_{s\zeta_F}(\RD^i(\RV))<\infty$    
    and 
    the function  
    $s\in C_k\mapsto ({1}/{n})\log\mathfrak{L}_{n,e,F}^\RV(s\zeta_F)$ is convex. By \cite[Theorem 10.8]{rockafellar1997convex}, for all $s\in C_k$ the limit
    $\lim_{n\to\infty}({1}/{n})\log\mathfrak{L}_{n,e,F}^\RV(s\zeta_F)$
    exists and the sequence $\{s\mapsto({1}/{n})\log\mathfrak{L}_{n,e,F}^\RV(s\zeta_F)\}_{n\in\mathbb{N}}$
    of functions on $C_k$ converges uniformly on each closed subset of $C_k$. Therefore, since $p_F$ is continuous on $Fin$, if we set $\preintervalF':=\bigcap_{k\in\mathbb{N}}\RS_k$ then $\preintervalF'$ satisfies desired conditions.
\end{proof}

Let $\Phi$ be a normal RCGDMS and let $F$ be a subset of $E$ such that $\CF$ is finitely primitive. For the full-measure set $\preintervalF'$ obtained in Lemma \ref{lemma convexity of the toporogical pressure} we set
\[
\preintervalF:=\bigcap_{k\in\mathbb{Z}}\RD^{k}\left(\preintervalF'\right).
\]
Since $\RD$ is invertible and $\RP$-preserving, we have $\RP(\preintervalF)=1$.
Since for all $\RV\in \lowb$ and $s\in\mathbb{R}$ we have $\underline{M}_{s\zeta_F}(\RV)\geq \min\{\cc^s,\min\{M_e^s:e\in\Lambda^F_a\}\}$, there exists $\tailR\subset\RS$
 such that $\RP\left(\tailR\right)=1$, 
    and for all $k\in\mathbb{Z}$, $\RV\in\tailR$ and $s\in  Fin_F$ we have
\begin{align}\label{eq tail zero}
    \lim_{n\to\infty}\frac{1}{n}\log M_{s\zeta_F}(\RV_{k+n})=0.
\end{align}
     If $F=E$ then we will write $\preinterval=\preintervalF$ and $\tailRE=\tailR$.

For a finite subset $F$ of $E$ such that $\CF$ is finitely primitive and all $s\in \mathbb{R}$ we will denote by $\{\widetilde m_{F,\RV}^s\}_{\RV\in\RS}$ random conformal measures with respect to the potential $s\zeta_F$ obtained in Theorem \ref{thm conformal measure}
and for all $\RV\in\RS$ we define the Borel probability measure $m_{F,\RV}^s$ on $J_{\RV,F}$ by $m_{F,\RV}^s:=\widetilde m_{F,\RV}^s\circ \pi_{\RV,F}^{-1}$. For all $n\in\mathbb{N}$ and $\RV\in\RS$ we set $P_{s\zeta_F}^n(\RV):=
\sum_{i=0}^{n-1}P_{s\zeta_F,{\widetilde m_{F,\RV_{i+1}}^s}}(\RV_i)
$ (see (\ref{eq abstruct eigenvalue})).

\begin{prop}\label{prop lyapunov spectrum and markov dimension}
    Let $\Phi$ be a normal RCGDMS satisfying  RBSC and let $F$ be a finite subset of $E$ such that  $\CF$ is finitely primitive. Then for all 
    $\RV\in \diR\cap\lowb\cap\finFPhi\cap\preintervalF\cap\tailR$,
    $s\in \mathbb{R}$ and $\beta>0$ we have 
\begin{align*}
     L_{\RV,F}(\beta)=
   K_{\RV,\mFs}^M\left(s+\frac{p_F(s)}{\beta}\right).
\end{align*}
   
\end{prop}
\begin{proof}
Let     $\RV\in \diR\cap\lowb\cap\finFPhi\cap\preintervalF\cap\tailR$ and let
    $s\in \mathbb{R}$. 
    Note that, since $\Phi$ satisfy RBSC, $\pi_\RV$ is bijection. Thus, by Lemma \ref{lemma gibbs} and (\ref{eq tail zero}), for all $\tau\in \CF$ we obtain
    \begin{align*}
\lim_{n\to\infty}
\frac{\log\mFs\left(\phi_{\tau|_n,\RV}\left(X_{t\left(\tau|_n\right)}\right)\right)}{\log\left|\phi_{\tau|_n,\RV}'(\pi_{\RV_{n+1}}(\sigma^{n+1}(\tau)))\right|}
=\lim_{n\to\infty}
\left(s+
\frac{(1/n) P_{s\zeta_F}^n(\RV)}{-(1/n)S_n\zeta_F(\tau,\RV)}
\right).
    \end{align*}
    if the limit exists. On the other hand, by Lemma \ref{lemma pointwise pressure and global pressure}, for all $s\in \mathbb{R}$ we have
    $\lim_{n\to\infty}({1}/{n})P_{s\zeta_F}^n(\RV)=p_F(s). $ 
    Hence, we obtain the desired result.
\end{proof}

\begin{lemma}\label{lemma basic properties of the pressure}
    Let $\Phi$ be a normal RCGDMS and let $F$ be a subset of $E$ such that $\CF$ is finitely primitive. 
    Then we have the following conditions:\\
 (P1) For $F\subset \tilde F$ such that $\tilde F_A^\infty$ is finitely primitive and $s\in\mathbb{R}$ we have $p_F(s)\leq p_{\tilde F}(s)$.\\
(P2) If $s_{\infty,F}\in\mathbb{R}$ then $p_F$ is right-continuous at $s_{\infty,F}$.\\
(P3)  $p_F$ is strictly decreasing on ${Fin}_F$.\\
(P4) We have
$\lim_{s\to\infty}p_F(s)=-\infty$
 and if $s_{\infty,F}=-\infty$ then
$\lim_{s\to-\infty}p_F(s)=\infty$.    
(P5) If $F$ is a finite set then $p_F$ is real-analytic on $\mathbb{R}$
\end{lemma}
\begin{proof}
(P1) follows from Proposition \ref{prop definition of the pressure}. (P2) follows from (P1) and Proposition \ref{prop compact approximation geometric pressure}.
By Theorem \ref{thm variational principle}, for $s_0\in Fin^\circ_F$ and $s_{\infty,F}<t<s_0<s$ we obtain
\[
\frac{p_F(s)-p_F(s_0)}{s-s_0}\leq 
{\log \cc} \text{ and }
\frac{p_F(t)-p_F(s_0)}{t-s_0}\leq \log \cc
.
\]
(P3) and (P4) follow from above inequalities.
We will show (P5) by using \cite[Theorem 9.17]{mayer2011distance}. We assume that $F$ is a finite set.  Let $s\in\mathbb{R}$ and let $D(s,1):=\{z\in\mathbb{C}:|z-s|<1\}$. Then for all $\RV\in \diR\cap\lowb$ and $\tau\in \CF$, $z\mapsto z\zeta_F(\tau,\RV)$ is holomorphic on $D(s,1)$, for all $s_0\in (s-1,s+1)$ we have $s_0(\zeta_F)_\RV\in H_\exponent^b(\CF)$ and for all $z\in D(s,1)$ we have $\|(z\zeta_F)_\RV\|_\exponent\leq (|s|+1)(\max\{-\log M_e:e\in F\}+v_\exponent(\zeta))$. Moreover, for all $z\in D(s,1)$ we have $\|\text{Im}((z\zeta_F)_\RV)\|_\exponent\leq|\text{Im}(z)|(\max\{-\log M_e:e\in F\}+v_\exponent(\zeta))$. Thus, by noting that $\RP(\diR\cap\lowb)=1$ and using \cite[Theorem 9.17]{mayer2011distance}, we obtain (P5).
\end{proof}

Let $F$ be a finite subset of $E$ such that  $\CF$ is finitely primitive.
We define the function $p_{F,0}:\mathbb{R}^2\rightarrow \mathbb{R}$ by $p_{F,0}(q,b):=P_F(-qp_F(0)+b\zeta_F)$.
Then for all $q\in\mathbb{R}$ there exists a unique $T_{F}(q)\in\mathbb{R}$ such that $p_{F,0}(q,T_{F}(q))=0$. By Lemma \ref{lemma convexity of the toporogical pressure} and Lemma \ref{lemma basic properties of the pressure}, the function $q\mapsto T_{F}(q)$ is convex and real-analytic (for example see \cite[Chapter 6.1]{mayer2011distance}).
We define the concave Legendre transform 
    $T_{F}^*$ of $T_{F}$ by
    \[
    T_{F}^*(\alpha):=\inf\{\alpha q+T_{F}(q) :{q\in\mathbb{R}}\}\ 
    \text{for $\alpha\in\mathbb{R}$}.
    \] 
For a convex function $\phi$ on $\mathbb{R}$ we define
$
\phi'(+\infty)=\lim_{x\to\infty}\phi_+'(x)
$
 and 
$\phi'(-\infty)=\lim_{x\to-\infty}\phi_+'(x)$,
where $\phi_+'(x)$ denotes the right-hand derivative of $\phi$ at $x$.
 The following result is shown in \cite[Theorem 6.5]{mayer2011distance}. After that \cite[Theorem 2.6]{Zhihui} show the following result in a more general setting. 
\begin{thm}
\label{thm multifractal analysis for local dimension}
Let $\Phi$ be a normal RCGDMS $\Phi$ satisfying RBSC and let $F$ be a finite subset of $E$ such that $\CF$ is finitely primitive.  
    Then there exists $\loc\subset \RS$ such that $\RP(\loc)=1$ and for all $\RV\in\loc$ and $\alpha\in (-T_{F}'(+\infty),-T_{F}'(-\infty))$ we have
    $g_{\RV,m_{F,\RV}^0}(\alpha)=T_{F}^*(\alpha).$
\end{thm}

For a RCGDMS $\Phi$ satisfying RBSC, a subset $F\subset E$ such that $\CF$ is finitely primitive, $\RV\in\RS$ and $\beta\in\mathbb{R}$ we define $L_{\RV,F}(\beta):=L_{\RV}(\beta)\cap J_{\RV,F}$ and $l_{\RV,F}(\beta):=\dim_H(L_{\RV,F}(\beta))$. 

\begin{thm}\label{thm multifractal analysis for lyapunov compact}
    Let $\Phi$ be a normal RCGDMS satisfying RBSC and let $F$ be a finite subset of $E$ such that $\CF$ is finitely primitive and $\#F\geq 2$. 
    Then there exists $\RS_{F,\text{Ly}}$ such that $\RP(\RS_{F,\text{Ly}})=1$ and 
    for all  $\RV\in\RS_{F,\text{Ly}}$ and $\beta\in (-p'_F(+\infty),-p'_F(-\infty))$ we have
    \[
    l_{\RV,F}(\beta)=\frac{1}{\beta}\inf_{s\in\mathbb{R}}\{p_F(s\zeta_F)+\beta s\}.
    \]
\end{thm}
\begin{proof}
Let $\RS_{F,\text{Ly}}:=\loc\cap\diR\cap \lowb\cap \finFPhi\cap\preintervalF$ and let $\RV\in\RS_{F,\text{Ly}}$.
Note that for all $(q,b)\in \mathbb{R}^2$ we have 
$p_{F,0}(q,b)=-qP_F(0)+p_F(b)$ and $ 
p_{F,0}(q,T_{F}(q))=0.    
$
By (P4) of Lemma \ref{lemma basic properties of the pressure}, we obtain 
$
\lim_{q\to\infty}T_{F}(q)=-\infty$ and 
$\lim_{q\to-\infty}T_{F}(q)=\infty .
$
In particular, we obtain 
 $   T_{F}(\mathbb{R})=\mathbb{R}.$
Moreover, by noting that $\#F\geq 2$ and Lemma \ref{lemma basic properties of the pressure}, for all $q\in\mathbb{R}$ we obtain
$T_{F}'(q)={P_F(0)}/{p_{F}'(T_{F}( q))}<0.$ 
Therefore, $x\in(-T'_{F}(+\infty),-T'_{F}(-\infty))$ if and only if $P_F(0)/x\in(-p'_F(+\infty),-p_F'(-\infty))$.
Let $\beta\in (-p_F'(+\infty),-p_F'(-\infty))$. By Proposition \ref{prop lyapunov spectrum and markov dimension}, Proposition \ref{prop markov dimension and local dimension} and Theorem \ref{thm multifractal analysis for local dimension}, we obtain 
\begin{align*}
&l_{\RV,F}(\beta)=g_{\RV,m_{F,\RV}^0}\left(\frac{P_F(0)}{\beta}\right)=
\inf
\left\{
\frac{P_F(0)}{\beta}q+T_{F}(q):q\in\mathbb{R}
\right\}
\\&
=\inf
\left\{
\frac{P_F(T_{F}(q)\zeta_F)}{\beta}+T_{F}(q):q\in\mathbb{R}
\right\}
=\frac{1}{\beta}\inf
\left\{
P_F(s\zeta_F)+\beta s:s\in\mathbb{R}
\right\}
.    
\end{align*}
\end{proof}

Let $\{F_n\}_{n\in\mathbb{N}}$ be a sequence of subsets of $E$ such that $\Lambda_a\subset F_1$, for all $n\in\mathbb{N}$ the set $F_n$ is a finite set, $F_n\subset F_{n+1}$ and $E=\bigcup_{n\in\mathbb{N}}F_n$. For all $n\in\mathbb{N}$  we set 
$p_n(s):=p_{F_n}(s).$
The following lemma is a direct consequence of Proposition \ref{prop compact approximation geometric pressure}, Lemma \ref{lemma convexity of the toporogical pressure} and (P1) of Lemma \ref{lemma basic properties of the pressure}.
\begin{lemma}\label{lemma range of the deribative of pressure}
    Let $\Phi$ be a normal RCGDMS  and
    let $\{F_n\}_{n\in\mathbb{N}}$ be a  sequence of subsets of $E$ such that $\Lambda_a\subset F_1$, for all $n\in\mathbb{N}$ the set $F_n$ is a finite set, $F_n\subset F_{n+1}$ and $E=\bigcup_{n\in\mathbb{N}}F_n$. For all $n\in \mathbb{N}$ we have
  $   -p'_{n+1}(+\infty)
    \leq
    -p'_{n}(+\infty)$
     ,
     $-p'_n(-\infty)
    \leq -p'_{n+1}(-\infty)$   
     and 
    $\lim_{n\to\infty}p'_{n}(+\infty)=p'(+\infty).$
Furthermore, if $\Phi$ is cofinitely regular then we have
$\lim_{n\to\infty} p_n'(s_{\infty})=-\infty=p'(s_{\infty}).$
\end{lemma}


\begin{lemma}\label{lemma compact apploximation of lgendre transform}
Let $\Phi$ be a normal cofinitely regular RCGDMS and
    let $\{F_n\}_{n\in\mathbb{N}}$ be an ascending  sequence of subsets of $E$ such that conditions of Lemma \ref{lemma range of the deribative of pressure} hold.
    Then, for all $\beta\in (-p'(+\infty),\infty)$ we have
\[
\lim_{n\to\infty}
\inf
\left\{
p_n(s)+\beta s:s\in\mathbb{R}
\right\}
=
\inf
\left\{
p(s)+\beta s:s\in\mathbb{R}
\right\}
.\]
\end{lemma}

\begin{proof}
Let $\beta\in (-p'(+\infty),\infty)\subset (0,\infty)$.
For all $n\in\mathbb{N}$ we define 
$F_n(\beta)
=
\inf
\left\{
p_n(s)+\beta s:s\in\mathbb{R}
\right\}
$
 and 
$F(\beta)
=
\inf
\left\{
p(s)+\beta s:s\in\mathbb{R}
\right\}.
$
By (P1) of Lemma \ref{lemma basic properties of the pressure},  it is enough to show that $F_{\infty}(\beta):=\lim_{n\to\infty}F_n(\beta)\geq F(\beta)$.
For a contradiction, we assume that $F_{\infty}(\beta)< F(\beta)$. We set $\epsilon:=(F(\beta)-F_{\infty}(\beta))/2>0$. Then, for all $n\in \mathbb{N}$ we have 
\begin{align}\label{eq proof of compact apploximation of transform}
    F_n(\beta)\leq F_{\infty}(\beta)<F(\beta)-\epsilon.
\end{align} 
For all $n\in\mathbb{N}$ we define 
$A_n:=\{s\in\mathbb{R}:p_n(s)+s\beta\leq 
F(\beta)-\epsilon
\}.$
By (P1) of Lemma \ref{lemma basic properties of the pressure}, for all $n\in\mathbb{N}$ we have $A_{n+1} \subset A_n$. 
Since for all $n\in\mathbb{N}$ the function $s\in\mathbb{R}\mapsto p_n(s)+s\beta$ is continuous, for all $n\in\mathbb{N}$ the set $A_n$ is closed.
By (\ref{eq proof of compact apploximation of transform}), for all $n\in\mathbb{N}$ the set $A_n$ is non-empty.
By Lemma \ref{lemma range of the deribative of pressure}, there exists $N\in\mathbb{N}$ such that for all $n\geq N$ we have
$-p_n'(+\infty)<\beta<-p_n'(-\infty)$.
Thus, for all $n\geq N$ we have
$
\lim_{s\to\infty}p_n(s)+\beta s=\infty
$
and
$
\lim_{s\to-\infty}p_n(s)+\beta s=\infty.
$
This implies that for all $n\geq N$  the set $A_n$ is a bounded set.
Thus, for all $n\geq N$ the set $A_n$ is a non-empty compact set satisfying $A_{n+1}\subset A_n$ and thus, $\bigcap_{n=N}^\infty A_n$ is non-empty.
By definition of $F(\beta)$ and Proposition \ref{prop compact approximation geometric pressure}, for $s\in \bigcap_{n=N}^\infty A_n$ we have $F(\beta)\leq p(s)+\beta s\leq F(\beta)-\epsilon$. This is a contradiction. Thus, we obtain $F(\beta)\leq F_\infty(\beta)$.
\end{proof}

\begin{prop}
    \label{prop lower bound of the lyapunov spectrum}
    Let $\Phi$ be a normal cofinitely regular RCGDMS $\Phi$ satisfying  RBSC. Then,
    there exists $\underline{\RS}\subset \RS$ such that $\RP(\underline{\RS})=1$ and for all $\RV\in\underline{\RS}$  and  $\beta\in (-p'(+\infty),\infty)$ we have 
    \[
    l_{\RV}(\beta)\geq\frac{1}{\beta}\inf
\left\{
p(s)+\beta s:s\in\mathbb{R}
\right\}.
    \]
\end{prop}
\begin{proof}
We consider an ascending sequence $\{F_n\}_{n\in\mathbb{N}}\subset E$ such that conditions of Lemma \ref{lemma range of the deribative of pressure} hold. Let $\RS_{F_n}$ ($n\in\mathbb{N}$) be the full-measure set obtained in Theorem \ref{thm multifractal analysis for lyapunov compact}. We set  $\underline{\RS}:=\bigcap_{n\in\mathbb{N}}\RS_{F_n}$.
Let $\beta\in (-p'(+\infty),\infty)$ and let $\RV\in \underline{\RS}$. By Lemma \ref{lemma range of the deribative of pressure}, there exists $N\in\mathbb{N}$ such that for all $n\geq N$ we have $\beta\in(-p'_n(+\infty),-p'_n(-\infty))$. Therefore, by  definitions of $l_{\RV,F_n}$ ($n\in\mathbb{N}$) and $l_{\RV}$, Theorem \ref{thm multifractal analysis for lyapunov compact} and Lemma \ref{lemma compact apploximation of lgendre transform}, we obtain
$l_{\RV}(\beta)\geq \lim_{n\to\infty}l_{\RV,F_n}(\beta)=({1}/{\beta})\inf \{p(s)+\beta s:s\in\mathbb{R}\}.$
\end{proof}

Let $\Phi$ be a RCGDMS. Note that for all $k\in\mathbb{Z}$, $s\in\mathbb{R}$ and $\RV\in\diR$ we have 
$v_\exponent((s\zeta)_{\RV_k})=sv_\exponent(\zeta_{\RV_k})\leq sv_\exponent(\zeta)=v_\exponent(s\zeta)$ and thus, $\RV\in \RS_{d,s\zeta}$.
  Therefore, by Theorem \ref{thm conformal measure} and (\ref{eq monotonicity}), for all $\RV\in \diR\cap\finEPhi$ and $s\in  Fin$ we obtain fiberwise multifractal measures $\{\widetilde m_{\RV_{k-1}}^{s\zeta}\}_{k\in\mathbb{N}}$ with respect to the potential $s\zeta$ and $\RV$. For $\RV\in \diR\cap\finEPhi$ and $s\in  Fin$ we denote by $\{\widetilde m_{\RV_{k-1}}^s\}_{k\in\mathbb{N}}$ fiberwise multifractal measures with respect to the potential $s\zeta$ and $\RV$, and we define the Borel probability measure $m^s_{\RV}$ on $J_{\RV}$ by $m^s_{\RV}:=\widetilde m_\RV^s\circ\pi_{\RV}^{-1}$.

\begin{prop}
\label{prop upper bound of the lyapunov spectrum}
    Let $\Phi$ be a normal RCGDMS $\Phi$ satisfying RBSC. Then
    there exists $\overline{\RS}\subset \RS$ such that $\RP(\overline{\RS})=1$ and for all $\RV\in\overline{\RS}$  and $\beta\in (-p'(+\infty),\infty)$ we have 
    \[
    l_{\RV}(\beta)\leq\frac{1}{\beta}\inf
\left\{
p(s)+\beta s:s\in\mathbb{R}
\right\}.
    \]    
\end{prop}
\begin{proof}
    We define $\overline{\RS}:=\diR\cap\lowb\cap\finEPhi\cap \preinterval\cap\tailRE$. 
Let $\RV\in\overline{\RS}$ and let $\beta\in (-p'(+\infty),\infty)$. 
We fix $x\in L_\RV(\beta)$, $s\in  Fin$ and $0<\delta<\beta$. We set $\tau:=\pi_\RV^{-1}(x)$. Since $\lim_{n\to\infty}(-1/n)S_n\zeta(\tau,\RV)=\beta$, there exists $N\geq 1$ such that for all $n\geq N$ we have
\begin{align*}
e^{-n(\beta+\delta)}\leq |\phi'_{\tau|_{n-1},\RV}(\pi_{\RV_{n}}(\sigma^{n}(\tau)))|\leq e^{-n(\beta-\delta)}.    
\end{align*}
By Lemma \ref{lemma bounded distortion} and Lemma \ref{lemma diameter condition}, for all $n\in\mathbb{N}$ we have
\[
\frac{1}{K_{bd}D_\Phi}
\leq \frac{|\phi'_{\tau|_{n-1},\RV}(\pi_{\RV_{n}}(\sigma^{n}(\tau)))|}{\text{diam}\left(\phi_{\tau|_{n-1},\RV}\left(X_{t(\tau_{n-1})}\right)\right)}
\leq K_{bd}D_\Phi.
\]
On the other hand, by Lemma \ref{lemma gibbs}, there exists a constant $C\geq 1$ such that for all $n\in\mathbb{N}$ we have 
\[
\frac{1}{CM_n}\leq \frac{\widetilde m_{\RV}^s([\tau|_{n-1}])}{|\phi'_{\tau|_{n-1},\RV}(\pi_{\RV_{n}}(\sigma^{n}(\tau)))|^s \exp(-P_{s\zeta}^n(\RV))}
\leq C,
\]
where $M_n:=\prod_{i=0}^{2N_A-1}M_{s\zeta}(\RV_{n+i})$.
Since $\Phi$ satisfy RBSC, for all $n\in\mathbb{N}$ we have $m^s_{\RV}(\phi_{\tau|_{n-1},\RV}(X_{t(\tau_{n-1})}))=\widetilde m_\RV^s\circ \pi_{\RV}^{-1}(\phi_{\tau|_{n-1},\RV}(X_{t(\tau_{n-1})}))=\widetilde m_\RV^s([\tau|_{n-1}])$.

We first consider the case where $p(s)\geq 0$. By above inequalities, 
for all $n\in\mathbb{N}$ we obtain
\begin{align*}
    &m^s_{\RV}(\phi_{\tau|_{n-1},\RV}(X_{t(\tau_{n-1})}))
    \geq 
    \frac{e^{-P_{s\zeta}^n(\RV)+np(s)}}{CM_n(K_{bd}D_\Phi)^s}e^{-np(s)}\left(\text{diam}\left(\phi_{\tau|_{n-1},\RV}\left(X_{t(\tau_{n-1})}\right)\right)\right)^s
    \\&\geq 
    \frac{e^{-P_{s\zeta}^n(\RV)+np(s)}}{CM_n(K_{bd}D_\Phi)^s}
    |\phi'_{\tau|_{n-1},\RV}(\pi_{\RV_{n}}(\sigma^{n}(\tau)))|^{p(s)/(\beta-\delta)}
    \left(\text{diam}\left(\phi_{\tau|_{n-1},\RV}\left(X_{t(\tau_{n-1})}\right)\right)\right)^s
    \\&
    \geq
        \frac{e^{-P_{s\zeta}^n(\RV)+np(s)}}{CM_n(K_{bd}D_\Phi)^{s+p(s)/(\beta-\delta)}}
    \left(\text{diam}\left(\phi_{\tau|_{n-1},\RV}\left(X_{t(\tau_{n-1})}\right)\right)\right)^{s+p(s)/(\beta-\delta)}.
    \end{align*}
    Hence, we obtain 
    \begin{align*}
        &\liminf_{r\to 0}\frac{\log m^s_{\RV}(B(x,r))}{\log r}
        \leq \liminf_{n\to\infty}\frac{\log m^s_{\RV}\left(B\left(x,\text{diam}\left(\phi_{\tau|_{n-1},\RV}\left(X_{t(\tau_{n-1})}\right)\right)\right)\right)}{\log\left(\text{diam}\left(\phi_{\tau|_{n-1},\RV}\left(X_{t(\tau_{n-1})}\right)\right)\right)}
        \\&\leq
s+\frac{p(s)}{\beta-\delta}
        +
        \liminf_{n\to\infty}\frac{-P^n_{s\zeta}(\RV)+np(s)-\log M_n}{\log \left(\text{diam}\left(\phi_{\tau|_{n-1},\RV}\left(X_{t(\tau_{n-1})}\right)\right)\right)}.
    \end{align*}
    We will show that 
    \[
    \liminf_{n\to\infty}\frac{-P^n_{s\zeta}(\RV)+np(s)-\log M_n}{\log \left(\text{diam}\left(\phi_{\tau|_{n-1},\RV}\left(X_{t(\tau_{n-1})}\right)\right)\right)}=0.
    \]
    Since $\RV\in \preinterval\cap \tailRE$, we have
    $p(s)=\lim_{n\to\infty}(1/n)P_{s\zeta}^n(\RV)$ and $\lim_{n\to\infty}(1/n)\log M_n=0$. Thus,
    for all $0<\epsilon<-\log \cc$ there exists $\widetilde N\geq 1$ such that for all $n\geq \widetilde N$ we have 
    $
    |-P^n_{s\zeta}(\RV)+np(s)|<n\epsilon,$ $\left|({1}/{n})\log M_n\right|<\epsilon$
     and 
    ${\log D_\Phi}/{n}<\epsilon
    .$
    Thus, by Lemma \ref{lemma diameter condition}, for all $0<\epsilon<-\log \cc$ and $n\geq \widetilde N$ we have
    \[
    \left|\frac{-P^n_{s\zeta}(\RV)+np(s)-\log M_n}{\log \left(\text{diam}\left(\phi_{\tau|_{n-1},\RV}\left(X_{t(\tau_{n-1})}\right)\right)\right)}\right|<\frac{2n\epsilon}{-\log D_\Phi-\log \left\|\phi'_{\tau|_{n-1},\RV}\right\|_{\infty}}\leq\frac{\epsilon}{-\log \cc}. 
    \]
    Thus, we obtain the desired result. Therefore, letting $\delta\to0$, we obtain
    \begin{align}\label{eq proof of the upper bound}
    \liminf_{r\to 0}\frac{\log m^s_{\RV}(B(x,r))}{\log r}\leq s+\frac{p(s)}{\beta}.    
    \end{align}
    As the same argument in the case $p(s)\geq0$, in the case $p(s)<0$, we can show the inequality (\ref{eq proof of the upper bound}).
Therefore, by \cite[Theorem 15.6.3 (b)]{urbanski2022non}, for all $s\in  Fin$ we obtain
$
l_\RV(\beta)\leq s+{p(s)}/{\beta}.
$
Since for $s\notin  Fin$ we have $s+{p(s)}/{\beta}=\infty$, we obtain 
$l_\RV(\beta)\leq ({1}/{\beta})\inf \left\{ p(s)+\beta s:s\in\mathbb{R}\right\}.$
\end{proof}

\begin{lemma}\label{lemma range of spectrum}
    Let $\Phi$ be a normal cofinitely regular RCGDMS satisfying  RBSC and let $F$ be a  subset of $E$ such that $\CF$ is finitely primitive. If $F$ is finite then for all $\RV\in\diR \cap\preintervalF$ we have
    $
    \left\{
    \beta\in\mathbb{R}:L_{\RV,F}(\beta)\neq \emptyset
    \right\}\subset [-p'_F(+\infty),-p'_F(-\infty)].
    $
    Moreover, for all $\RV\in\diR \cap\preinterval$ we have
    $
    \left\{
    \beta\in\mathbb{R}:L_{\RV}(\beta)\neq \emptyset
    \right\}\subset [-p'(+\infty),\infty).
    $
\end{lemma}
\begin{proof}
Since $p_F$ is convex on $ Fin_F$, we have $p_F'(+\infty)=\lim_{s\to\infty}p_F(s)/s$. Moreover, if $F$ is a finite set then we have $p_F'(-\infty)=\lim_{s\to-\infty}p_F(s)/s$.
Let $\RV\in\diR \cap\preintervalF$ and let $\beta\in \left\{
    \beta\in\mathbb{R}:L_{\RV,F}(\beta)\neq \emptyset
    \right\}$. Then there exists $\tau\in \CF$ such that
    $
    \lim_{n\to\infty}({1}/{n})S_n\zeta_F(\tau,\RV)=-\beta.
    $
    By Lemma \ref{lemma several form of prressure}, for all $s>s_{\infty,F}$ we have
    \[
    p_F(s)=
    \lim_{n\to\infty}\frac{1}{n}\log \sum_{a\in F^{n}_A}\exp(\sup\{sS_n\zeta_F(\tilde{a},\RV):\tilde{a}\in[a]_F\})
    \geq
    \lim_{n\to\infty}\frac{s}{n}S_n\zeta_F(\tau,\RV)
    .
    \]
Therefore, we obtain 
     $-p'_F(+\infty)=-\lim_{s\to\infty}{p_F(s)}/{s}\leq \beta$.
     Moreover, if $F$ is a finite set then
     $-p'_F(-\infty)=-\lim_{s\to-\infty}{p_F(s)}/{s}\geq \beta$.
\end{proof}

\emph {Proof of Theorem $\ref{thm main}$.}
Theorem \ref{thm main} is a direct consequence of Theorem \ref{thm multifractal analysis for lyapunov compact}, Proposition \ref{prop lower bound of the lyapunov spectrum}, Proposition \ref{prop upper bound of the lyapunov spectrum} and Lemma \ref{lemma range of spectrum}. \qed

\subsection*{Acknowledgments}
This work was supported by the JSPS KAKENHI 25KJ1382.

\bibliographystyle{abbrv}
\bibliography{reference}
 \nocite{*}

\end{document}